\documentclass{article}

\usepackage{amsthm}
\usepackage{amsmath, amssymb}
\usepackage{hyperref}
\usepackage{tikz}
\usetikzlibrary{matrix}
\usepackage[matrix, arrow, curve]{xy}

\newtheorem{Thm}{Theorem}[section]
\newtheorem{Def}[Thm]{Definition}
\newtheorem{Lem}[Thm]{Lemma}
\newtheorem{Prop}[Thm]{Proposition}
\newtheorem{Kor}[Thm]{Corollary}

\title{A Generalized Vaserstein Symbol}

\author{Tariq Syed  \\
	Fakult\"at f\"ur Mathematik\\
	Universit\"at Duisburg-Essen \\
	Thea-Leymann-Stra{\ss}e 9 \\
	D-45127 Essen \\
	tariq.syed@gmx.de}

\date{\today} %\date{}

\begin{document}

\maketitle

\begin{abstract}
Let $R$ be a commutative ring. For any projective $R$-module $P_0$ of constant rank $2$ with a trivialization of its determinant, we define a generalized Vaserstein symbol on the orbit space of the set of epimorphisms $P_0 \oplus R \rightarrow R$ under the action of the group of elementary automorphisms of $P_0 \oplus R$, which maps into the elementary symplectic Witt group. We give criteria for the surjectivity and injectivity of the generalized Vaserstein symbol and deduce that it is an isomorphism if $R$ is a regular Noetherian ring of dimension $2$ or a regular affine algebra of dimension $3$ over a perfect field $k$ with $c.d.(k) \leq 1$ and $6 \in k^{\times}$.
\end{abstract}

\tableofcontents

\section{Introduction}

In this paper, we provide a generalized construction of the Vaserstein symbol, which was originally introduced by Andrei Suslin and Leonid Vaserstein in \cite{SV}. We let $R$ be a commutative ring and we let $\mathit{Um}_{n}(R)$ denote the set of unimodular rows of length $n$, i.e. row vectors $(a_{1}, a_{2}, , ..., a_{n})$ such that $\langle a_{1}, a_{2}, ..., a_{n} \rangle = R$. Such row vectors obviously correspond to epimorphisms $R^{n} \rightarrow R$. Therefore the group $GL_{n} (R)$ of invertible $n\times n$-matrices acts on the right on $\mathit{Um}_{n}(R)$ (by precomposition); consequently the same holds for any subgroup of $GL_{n} (R)$, e.g. the group $SL_{n} (R)$ of invertible $n\times n$-matrices with determinant $1$ or its subgroup $E_{n} (R)$ generated by elementary matrices. Note that the set $\mathit{Um}_{n}(R)$ has a canonical basepoint given by the row $e_{1} = (1, 0,...,0)$.\\
Now let $n=3$ and let $(a_1, a_2, a_3)$ be a unimodular row of length $3$. By definition, there exist elements $b_1,b_2,b_3 \in R$ such that $\sum_{i=1}^{3} a_i b_i = 1$. Therefore the alternating matrix

\begin{center}
$V (a,b) = \begin{pmatrix}
0 & - a_1 & - a_2 & - a_3 \\
a_1 & 0 & - b_3 & b_2 \\
a_2 & b_3 & 0 & - b_1 \\
a_3 & - b_2 & b_1 & 0
\end{pmatrix}$
\end{center}

has Pfaffian $1$ and represents an element of the so-called elementary symplectic Witt group $W_{E} (R)$. It was shown in \cite[Lemma 5.1]{SV} that this element is independent of the choice of the elements $b_1,b_2,b_3$. Furthermore, it follows from \cite[Theorem 5.2(a)]{SV} that this assignment is invariant under the action of $E_{3} (R)$ on $\mathit{Um}_{3}(R)$. Therefore one obtains a well-defined map

\begin{center}
$V: \mathit{Um}_{3}(R)/E_{3} (R) \rightarrow W_{E} (R)$
\end{center}

called the Vaserstein symbol. Suslin and Vaserstein also found criteria for this map to be surjective or injective in terms of the right action of $E_{n}(R)$ on $\mathit{Um}_{n}(R)$ mentioned above. More precisely, they proved that the Vaserstein symbol is surjective if $\mathit{Um}_{2n+1}(R) = e_{1} E_{2n+1}(R)$ for $n \geq 2$ (cf. \cite[Theorem 5.2(b)]{SV}) and injective if $e_{1} E_{2n} = e_{1} (E(R) \cap GL_{2n}(R))$ for $n \geq 3$ and $E(R) \cap GL_{4}(R) = E_{4}(R)$ (combine \cite[Theorem 5.2(c)]{SV} and the proof of \cite[Corollary 7.4]{SV}).\\
These criteria immediately enabled them to deduce that the Vaserstein symbol is a bijection for a Noetherian ring of Krull dimension $2$ (cf. \cite[Corollary 7.4]{SV}). Using local-global principles, Rao and van der Kallen could prove in \cite[Corollary 3.5]{RvdK} that the Vaserstein symbol is also a bijection for a $3$-dimensional regular affine algebra over a field $k$ with $c.d. (k) \leq 1$, which is supposed to be perfect if $char (k) = 2,3$.\\
The Vaserstein symbol plays an important role in the study of stably free modules of rank $2$ (cf. \cite{HB2}, \cite{F}). Indeed, the orbit space $\mathit{Um}_{3}(R)/E_{3} (R)$ naturally surjects onto the set of isomorphism classes of projective $R$-modules of rank $2$ which become free after adding a free direct summand of rank $1$ (cf. Section \ref{2.4}). In \cite[Theorem 7.5]{FRS}, the Vaserstein symbol was crucially used in order to prove that stably free modules of rank $d-1$ over smooth affine $k$-algebras of dimension $d \geq 3$ are free whenever $k$ is algebraically closed and $(d-1)! \in k^{\times}$: By reducing to the case of a threefold and by using the result of Rao and van der Kallen mentioned in the previous paragraph, it was proven that any unimodular row of length $d$ can be transformed via elementary matrices to a row of the form $(a_{1},a_{2},...,a_{d}^{(d-1)!})$. Then Suslin's theorem that any such row can be completed to an invertible matrix (cf. \cite[Theorem 2]{S2}) implied the result.\\
While projective modules of rank $\geq d$ are cancellative in the situation of \cite[Theorem 7.5]{FRS}, the same is not true in general for projective modules of rank $d-2$ (cf. \cite{NMK}). In particular, stably isomorphic projective modules of rank $2$ over smooth affine fourfolds over algebraically closed fields need not be isomorphic in general.\\
Our work on the generalization of the Vaserstein symbol is substantially motivated by the study of projective modules as described in the previous paragraphs: The generalized Vaserstein symbol will lead to a conceptual explanation for the failure of the cancellation property of projective modules of rank $2$ with trivial determinant over smooth affine fourfolds over algebraically closed fields. By generalizing the approach in \cite{FRS}, we also foresee that the generalized Vaserstein symbol will be an important tool in order to study the cancellation property of projective modules of rank $d-1$ with trivial determinant over smooth affine algebras of dimension $d$ over an algebraically closed field $k$ with $(d-1)! \in k^{\times}$. To keep the length of this paper reasonable, the discussion of these major applications is deferred to subsequent work. Our results in this paper are as follows:\\
First, recall from \cite{SV} that the elementary symplectic Witt group $W_{E} (R)$ is defined as a subgroup of a larger group usually denoted $W'_{E} (R)$, which we will define in Section \ref{3.1}. The group $W'_E (R)$ is generated by alternating invertible matrices and $W_{E} (R)$ then corresponds to its subgroup generated by matrices with Pfaffian $1$. It is known that the group $W'_{E} (R)$ is isomorphic to the higher Grothendieck-Witt group $GW_{1}^{3} (R)$ and also to the group $V (R)$ (cf. \cite{FRS}, see Section \ref{3.2} below). The latter group is generated by isometry classes of triples $(P, g, f)$, where $P$ is a finitely generated projective $R$-module and $f$ and $g$ are alternating isomorphisms on $P$ (or, equivalently, non-degenerate alternating forms on $P$). Under the isomorphism $W'_{E} (R) \cong V (R)$, the group $W_{E} (R)$ then corresponds to a subgroup of $V (R)$. We denote this subgroup by $\tilde{V} (R)$.\\
Now let $P_{0}$ be a finitely generated projective $R$-module of rank $2$ with a fixed trivialization $\theta_{0}: R \xrightarrow{\cong} \det (P_{0})$ of its determinant. We denote by $Um (P_{0} \oplus R)$ the set of epimorphisms $P_{0} \oplus R \rightarrow R$ and by $E (P_{0} \oplus R)$ the group of elementary automorphisms of $P_{0} \oplus R$. Any element $a: P_{0} \oplus R \rightarrow R$ of $Um (P_{0} \oplus R)$ has a section $s: R \rightarrow P_{0} \oplus R$, which canonically induces an isomorphism $i: P_{0} \oplus R \xrightarrow{\cong} P(a) \oplus R$, where $P(a) = \ker(a)$. We let $\chi_{0}$ be the alternating form on $P_{0}$ which sends a pair $(p,q)$ to the element $\theta_{0}^{-1} (p \wedge q)$ of $R$; similarly, there is an isomorphism $\theta: R \xrightarrow{\cong} \det (P(a))$ obtained as the composite of $\theta_{0}$ and the isomorphism $\det (P_{0}) \cong \det (P(a))$ induced by $a$ and $s$. We then denote by $\chi_{a}$ the alternating form on $P(a)$ which sends $(p,q)$ to the element $\theta^{-1} (p \wedge q)$ of $R$. We then consider the element

\begin{center}
$V (a) = [P_{0} \oplus R^{2}, \chi_{0} \perp \psi_{2}, {(i \oplus 1)}^{t} (\chi_{a} \perp \psi_{2}) (i \oplus 1)]$
\end{center}

of $V (R)$. Our first result is the following (Theorem \ref{T4.1}, Lemma \ref{L4.2} and Theorem \ref{T4.3} in the text):
\\ \\
\textbf{Theorem 1.} The element $V(a)$ is independent of the choice of a section $s$ of $a \in Um (P_{0} \oplus R)$ and is an element of $\tilde{V} (R)$. Furthermore, we have $V (a) = V (a \varphi)$ in $V (R)$ for all $a \in Um (P_{0} \oplus R)$ and $\varphi \in E (P_{0} \oplus R)$. Thus, the assignment above descends to a well-defined map $V: Um (P_{0} \oplus R)/E (P_{0} \oplus R) \rightarrow \tilde{V} (R)$, which we call the generalized Vaserstein symbol (associated to the trivialization $\theta_{0}$ of $\det (P_{0})$).
\\ \\
The terminology is justified by the following observation: If we take $P_{0} = R^{2}$ and let $e_{1} = (1,0)$ and $e_{2} = (0,1)$, then it is well-known that there is a canonical isomorphism $\theta_{0}: R \xrightarrow{\cong} \det(R^{2})$ given by $1 \mapsto e_{1} \wedge e_{2}$. Then the generalized Vaserstein symbol associated to $- \theta_{0}$ coincides with the usual Vaserstein symbol via the identification $\tilde{V} (R) \cong W_{E} (R)$ mentioned above.\\
Of course, any two trivializations of $\det (P_{0})$ are equal up to multiplication by a unit of $R$. We will actually make precise how the generalized Vaserstein symbol depends on the choice of a trivialization of $\det (P_{0})$ by means of a canonical $R^{\times}$-action on $V (R)$.\\
We also generalize the criteria found by Suslin and Vaserstein on the injectivity and surjectivity of the Vaserstein symbol. For this, let $P_{n} = P_{0} \oplus R^{n-2}$ for all $n \geq 3$ and let $E_{\infty} (P_{0})$ be the direct limit of the groups $E (P_{n})$ for $n \geq 3$. Note that $Um (P_{n})$ has a canonical basepoint given by the projection $\pi_{n,n}$ onto the "last" free direct summand of rank $1$. We then prove (Theorem \ref{T4.5} and Theorem \ref{T4.14} in the text):
\\ \\
\textbf{Theorem 2.} The Vaserstein symbol $V: Um (P_{0} \oplus R)/E (P_{0} \oplus R) \rightarrow \tilde{V} (R)$ is surjective if $Um (P_{2n+1}) = \pi_{2n+1,2n+1} E (P_{2n+1})$ for all $n \geq 2$. Furthermore, it is injective if $\pi_{2n,2n} E (P_{2n}) = \pi_{2n,2n} (E_{\infty} (P_{0}) \cap Aut (P_{2n}))$ for all $n \geq 3$ and $E_{\infty} (P_{0}) \cap Aut (P_{4}) = E (P_{4})$.
\\ \\
Using local-global principles for transvection groups (cf. \cite{BBR}), we may then prove the following result (Theorems \ref{T2.15}, \ref{T2.16} and \ref{T4.15} in the text):
\\ \\
\textbf{Theorem 3.} The equality $E_{\infty} (P_{0}) \cap Aut (P_{4}) = E (P_{4})$ holds if $R$ is a $2$-dimensional regular Noetherian ring or if $R$ is a $3$-dimensional regular affine algebra over a perfect field $k$ such that $c.d.(k) \leq 1$ and $6 \in k^{\times}$. In particular, it follows that the generalized Vaserstein symbol $V: Um (P_{0} \oplus R)/E (P_{0} \oplus R) \rightarrow \tilde{V} (R)$ is a bijection in these cases.
\\ \\
As indicated above, the corresponding result by Rao and van der Kallen for the usual Vaserstein symbol in dimension $3$ and Suslin's theorem on the completability of unimodular rows were crucially used in the proof of \cite[Theorem 7.5]{FRS}. As a special case of Suslin's theorem, one obtains that unimodular rows of the form $(a_{1}, a_{2}, a_{3}^{2})$ are completable to invertible matrices. An explicit completion of such a unimodular row is given e.g. in \cite{Kr}. In fact, we can translate this result to our setting: Any epimorphism $a: P_{0} \oplus R \rightarrow R$ can be written as $(a_{0}, a_{R})$, where $a_{0}$ and $a_{R}$ are the restrictions of $a$ to $P_{0}$ and $R$ respectively. Then we can generalize Krusemeyer's construction in order to give an explicit completion of an epimorphism of the form $a= (a_{0}, a_{R}^{2})$ to an automorphism of $P_{0} \oplus R$ (see Proposition \ref{P4.18}).\\
While it should be possible to define a Vaserstein symbol without the assumption of a trivial determinant of $P_{0}$, it is by no means obvious that our methods in this paper can be adjusted in order to prove the same results without this assumption.\\
The organization of the paper is as follows: In Section \ref{2}, we basically prove the technical ingredients for the proofs of the main results of this paper. In particular, we prove some lemmas on elementary automorphisms of projective modules and use local-global principles for transvection groups in order to derive stability results for automorphism groups of projective modules. Section \ref{3} basically covers the definition of the elementary symplectic Witt group $W_{E} (R)$ and the identifications of $W'_{E} (R)$, $V(R)$ and $GW_{1}^{3} (R)$. In Section \ref{4}, we motivate and give the definition of the generalized Vaserstein symbol and begin to study its basic properties. We will then use all the technical lemmas proven in the previous sections in order to deduce the theorems stated above.

\subsection*{Notation and Conventions}

In this paper, a ring $R$ will always be commutative with unit. If $k$ is a perfect field, we will denote by $\mathcal{H} (k)$ the $\mathbb{A}^{1}$-homotopy category as defined by Morel and Voevodsky and by $\mathcal{H}_{\bullet} (k)$ its pointed version. If $\mathcal{X}$ and $\mathcal{Y}$ are spaces (resp. pointed spaces), we will write $[\mathcal{X},\mathcal{Y}]_{\mathbb{A}^{1}}$ (resp. $[\mathcal{X},\mathcal{Y}]_{\mathbb{A}^{1}, \bullet}$) for the set of morphisms from $\mathcal{X}$ to $\mathcal{Y}$ in $\mathcal{H} (k)$ (resp. $\mathcal{H}_{\bullet} (k)$).

\subsection*{Acknowledgements}

The author would like to thank his PhD advisor Andreas Rosenschon for a general introduction to the subject of his thesis, for many helpful discussions and for his support. The author would also like to thank his PhD advisor Jean Fasel for suggesting to work on the generalization of the Vaserstein symbol, for providing the main idea of this work and for many helpful discussions. Furthermore, he would like to thank Anand Sawant for many helpful discussions around algebraic and Hermitian K-theory and motivic homotopy theory and Ravi Rao for helpful comments on the local-global principle for transvection groups. Finally, the author would like to thank the anonymous referees for suggesting changes which helped to avoid the assumption that $2 \in R^{\times}$ in an earlier version of the paper and which greatly improved the exposition.\\
The author was funded by a grant of the DFG priority program 1786 "Homotopy theory and algebraic geometry".

\section{Preliminaries on projective modules}\label{Preliminaries on projective modules}\label{2}

In this section, we recall some basic facts on projective modules over commutative rings and prove some technical lemmas on elementary automorphisms which will be crucially used in the proofs of the main results of this paper. We also recall the local-global principle for transvection groups in \cite{BBR} in order to prove a stability result on automorphisms of projective modules. At the end of this section, we briefly recall how projective modules stably isomorphic to a given projective module $P$ can be classified in terms of the orbit space of the set of epimorphisms $P \oplus R \rightarrow R$ under the action of the groups of automorphisms of $P \oplus R$.

\subsection{Local trivializations and alternating isomorphisms on projective modules}\label{Local trivializations and alternating isomorphisms on projective modules}

Let $R$ be a commutative ring and $P$ be any projective $R$-module. For any prime ideal $\mathfrak{p}$ of $R$, the localized $R_{\mathfrak{p}}$-module $P_{\mathfrak{p}}$ is again projective and therefore free (because projective modules over local rings are free). In this weak sense, projective modules are locally free. If the rank of $P_{\mathfrak{p}}$ as an $R_{\mathfrak{p}}$-module is finite for every prime $\mathfrak{p}$, then we say that $P$ is a projective module of finite rank. In this case, there is a well-defined map $rank_P: Spec (R) \rightarrow \mathbb{Z}$ which sends a prime ideal $\mathfrak{p}$ of $R$ to the rank of $P_{\mathfrak{p}}$ as an $R_{\mathfrak{p}}$-module. It is not true in general that projective modules of finite rank are finitely generated; nevertheless, this is true if $rank_P$ is a constant map (cf. \cite[Chapter I, Exercise 2.14]{W}). We will say that $P$ is locally free of finite rank (in the strong sense) if it admits elements $f_1,..., f_n  \in R$ generating the unit ideal such that the localizations $P_{f_k}$ are free $R_{f_k}$-modules of finite rank. In fact, it is well-known that this is true if and only if $P$ is a finitely generated projective module. The following lemma follows from \cite[Chapter I, Lemma 2.4]{W} and \cite[Chapter I, Ex. 2.11]{W}:

\begin{Lem}\label{L2.1}
Let $R$ be a ring and $M$ be an $R$-module. Then the following statements are equivalent:
\begin{itemize}
\item[a)] $M$ is a finitely generated projective $R$-module;
\item[b)] $M$ is locally free of finite rank (in the strong sense);
\item[c)] $M$ is a finitely presented $R$-module and $M_{\mathfrak{p}}$ is a free $R_{\mathfrak{p}}$-module for every prime ideal $\mathfrak{p}$ of $R$;
\item[d)] $M$ is a finitely generated $R$-module, $M_{\mathfrak{p}}$ is a free $R_{\mathfrak{p}}$-module for every prime ideal $\mathfrak{p}$ of $R$ and the induced map $rank_M: Spec (R) \rightarrow \mathbb{Z}$ is continuous.
\end{itemize}
\end{Lem}

For any projective $R$-module $P$ of finite rank, there is a canonical isomorphism 

\begin{center}
$can: P \rightarrow P^{\vee\vee}$, $p \mapsto (ev_{p}: P^{\vee} \rightarrow R, a \mapsto a(p))$
\end{center}

induced by evaluation. A symmetric isomorphism on $P$ is an isomorphism $f: P \rightarrow P^{\vee}$ such that the diagram

\begin{center}
$\begin{xy}
  \xymatrix{
      P \ar[r]^{f} \ar[d]_{can}    &   P^{\vee} \ar@2{-}[d]^{id}  \\
      P^{\vee \vee} \ar[r]_{f^{\vee}}             &   P^{\vee}   
  }
\end{xy}$
\end{center}

\noindent is commutative. Similarly, a skew-symmetric isomorphism on $P$ is an isomorphism $f: P \rightarrow P^{\vee}$ such that the diagram

\begin{center}
$\begin{xy}
  \xymatrix{
      P \ar[r]^{f} \ar[d]_{-can}    &   P^{\vee} \ar@2{-}[d]^{id}  \\
      P^{\vee\vee} \ar[r]_{f^{\vee}}             &   P^{\vee}   
  }
\end{xy}$
\end{center}

\noindent is commutative. Finally, an alternating isomorphism on $P$ is an isomorphism $f: P \rightarrow P^{\vee}$ such that $f(p)(p) = 0$ for all $p \in P$.\\
A symmetric form on a projective $R$-module $P$ of finite rank is an $R$-bilinear map $\chi : P \times P \rightarrow R$ such that $\chi (p,q) = \chi (q,p)$ for all $p,q \in P$. Similarly, a skew-symmetric form on a projective $R$-module $P$ of finite rank is an $R$-bilinear map $\chi: P \times P \rightarrow R$ such that $\chi (p,q) = - \chi (q,p)$ for all $p,q \in P$. Moreover, an alternating form on a projective $R$-module $P$ of finite rank is an $R$-bilinear map $\chi: P \times P \rightarrow R$ such that $\chi (p,p) = 0$ for all $p \in P$. Note that any alternating form on $P$ is automatically skew-symmetric. If $2 \in R^{\times}$, any skew-symmetric form is alternating as well. A (skew-)symmetric or alternating form $\chi$ is non-degenerate if the induced map $P \rightarrow P^{\vee}, q \mapsto (p \mapsto \chi(p,q))$ is an isomorphism. Obviously, the data of a non-degenerate (skew-)symmetric form is equivalent to the data of a (skew-)symmetric isomorphism. Analogously, the data of a non-degenerate alternating form is equivalent to the data of an alternating isomorphism.\\
Now let $\chi: M \times M \rightarrow R$ be any $R$-bilinear form on $M$. This form induces a homomorphism $M \otimes_{R} M \rightarrow R$. For any prime $\mathfrak{p}$ of $R$, there is an induced homomorphism $M_{\mathfrak{p}} \otimes_{R_\mathfrak{p}} M_{\mathfrak{p}} \cong {(M \otimes_R M)}_{\mathfrak{p}} \rightarrow R_{\mathfrak{p}}$. This gives an $R$-bilinear form $\chi_{\mathfrak{p}} : M_{\mathfrak{p}} \times M_{\mathfrak{p}} \rightarrow R_{\mathfrak{p}}$ on $M_{\mathfrak{p}}$. The following lemma shows that these localized forms completely determine $\chi$:
 
\begin{Lem}\label{L2.2}
If $\chi_1$ and $\chi_2$ are $R$-bilinear forms on an $R$-module $M$. Then $\chi_1 = \chi_2$ if and only if ${\chi_1}_{\mathfrak{p}} = {\chi_2}_{\mathfrak{p}}$ for every prime ideal $\mathfrak{p}$ of $R$.
\end{Lem}
 
\begin{proof}
The forms $\chi_1$ and $\chi_2$ agree if and only if $\chi_1 (p,q) - \chi_2 (p,q) = 0$ for all $p,q \in M$. Therefore the lemma follows immediately from the fact that being $0$ is a local property for elements of any $R$-module.
\end{proof}

\subsection{Elementary automorphisms and unimodular elements}\label{Elementary automorphisms and unimodular elements}\label{2.2}

Again, let $R$ be a ring and let $M \cong \bigoplus_{i=1}^{n} M_{i}$ be an R-module which admits a decomposition into a \mbox{direct} sum of R-modules $M_{i}$, $i=1,...,n$. An elementary automorphism $\varphi$ of $M$ with respect to the given decomposition is an endomorphism of the form $\varphi_{s_{ij}} = id_M + s_{ij}$, where $s_{ij}: M_{j} \rightarrow M_{i}$ is an R-linear homomorphism for some $i \neq j$ (cf. \cite[Chapter IV, \S 3]{HB}). Any such homomorphism automatically is an isomorphism with inverse given by $\varphi_{s_{ij}}^{-1} = id_M - s_{ij}$. For $M=R^{n} \cong \bigoplus_{i=1}^{n} R$, one just obtains the automorphisms given by elementary matrices. We denote by $Aut (M)$ the group of automorphisms of $M$ and by $E(M_1,..., M_n)$ (or simply $E(M)$ if the decomposition is understood) the subgroup of $Aut(M)$ generated by elementary automorphisms.\\
The following lemma gives a list of useful formulas which can be checked easily by direct computation:
 
\begin{Lem}\label{L2.3}
Let $M = \bigoplus_{i=1}^{n} M_{i}$ be a direct sum of $R$-modules. Then we have
 
\begin{itemize}
\item[a)] $\varphi_{s_{ij}} \varphi_{t_{ij}} = \varphi_{(s_{ij} + t_{ij})}$ for all $s_{ij}: M_{j} \rightarrow M_{i}$, $t_{ij}: M_{j} \rightarrow M_{i}$ and $i \neq j$;
\item[b)] $\varphi_{s_{ij}} \varphi_{s_{kl}} = \varphi_{s_{kl}} \varphi_{s_{ij}}$ for all $s_{ij}: M_{j} \rightarrow M_{i}$, $s_{kl}: M_{l} \rightarrow M_{k}$, $i \neq j$, $k \neq l$, $j \neq k$ and $i \neq l$;
\item[c)] $\varphi_{s_{ij}} \varphi_{s_{jk}} \varphi_{- s_{ij}} \varphi_{- s_{jk}} = \varphi_{(s_{ij} s_{jk})}$ for all $s_{ij}: M_{j} \rightarrow M_{i}$, $s_{jk}: M_{k} \rightarrow M_{j}$ and distinct $i,j,k$;
\item[d)] $\varphi_{s_{ij}} \varphi_{s_{ki}} \varphi_{- s_{ij}} \varphi_{- s_{ki}} = \varphi_{(- s_{ki} s_{ij})}$ for all $s_{ij}: M_{j} \rightarrow M_{i}$, $s_{ki}: M_{i} \rightarrow M_{k}$ and distinct $i,j,k$.
\end{itemize}
\end{Lem}

If we restrict to the case $M_i = M_n$ for $i \geq 2$, we obtain the following result on $E (M)$:

\begin{Kor}\label{C2.4}
If $M_i = M_n$ for $i \geq 2$, then the group $E (M)$ is generated by the elementary automorphisms of the form $\varphi_{s} = id_M + s$, where $s$ is an $R$-linear map $M_i \rightarrow M_n$ or $M_n \rightarrow M_i$ for some $i \neq n$. The same statement holds if one replaces $n$ by any other $k \geq 2$.
\end{Kor}

\begin{proof}
Since $M_{i} = M_{n}$ for all $i \geq 2$, we have identities ${id}_{in}: M_{n} \rightarrow M_{i}$ and ${id}_{ni}: M_{i} \rightarrow M_{n}$ for all $i \geq 2$. Let $s_{ij}: M_{j} \rightarrow M_{i}$ be a morphism with $i \neq j$ and therefore either $i \geq 2$ or $j \geq 2$. We may assume that $i,j,n$ are distinct. If $i \geq 2$, then
 
\begin{center}
$\varphi_{s_{ij}} = \varphi_{id_{in}} \varphi_{id_{ni} s_{ij}} \varphi_{- id_{in}} \varphi_{(-{id}_{ni} s_{ij})}$
\end{center}
 
by the third formula in Lemma \ref{L2.3}. If $j \geq 2$, then
 
\begin{center}
$\varphi_{s_{ij}} = \varphi_{(s_{ij} {id}_{jn})} \varphi_{id_{nj}} \varphi_{(- s_{ij} {id}_{jn})} \varphi_{- id_{nj}}$.
\end{center}
 
by the third formula in Lemma \ref{L2.3}. This proves the first part of the corollary. The last part follows in the same way if $n$ is replaced by $k \geq 2$.
\end{proof}

The proof of Corollary \ref{C2.4} also shows:
 
\begin{Kor}\label{C2.5}
Let $M = \bigoplus_{i=1}^{n} M_{i}$ be a direct sum of $R$-modules and also let $s: M_{j} \rightarrow M_{i}$, $i \neq j$, be an $R$-linear map. Assume that there is $k \neq i$ with $M_{k} = M_{i}$ or $k \neq j$ with $M_{k} = M_{j}$. Then the induced elementary automorphism $\varphi_{s}$ is a commutator.
\end{Kor}

The following lemma is a version of Whitehead's lemma in our general setting:

\begin{Lem}\label{L2.6}
Let $M = M_{1} \oplus M_{2}$ and let $f: M_{1} \rightarrow M_{2}$, $g: M_{2} \rightarrow M_{1}$ be morphisms. Assume that $id_{M_{1}} + g f$ is an automorphism of $M_{1}$. Then:
\begin{itemize}
\item $id_{M_{2}} + f g$ is an automorphism of $M_{2}$ and
\item $(id_{M_{1}} + g f) \oplus {(id_{M_{2}} + f g)}^{-1}$ is an element of $E (M_{1} \oplus M_{2})$
\end{itemize}
\end{Lem}

\begin{proof}
We have $id_{M_{1}} \oplus (id_{M_{2}} + f g) = {\varphi}_{-f} {\varphi}_{-g} {((id_{M_{1}} + g f) \oplus id_{M_{2}})} {\varphi}_{f} {\varphi}_{g}$. This shows the first statement. For the second statement one checks that

\begin{center}
$(id_{M_{1}} + g f) \oplus {(id_{M_{2}} + f g)}^{-1} = {\varphi}_{-g} {\varphi}_{-f} {\varphi}_{g} {\varphi}_{{(id_{M_{1}} + g f)}^{-1} g - g} {\varphi}_{f g f + f}$.
\end{center}

So $(id_{M_{1}} + g f) \oplus {(id_{M_{2}} + f g)}^{-1}$ lies in $E (M_{1} \oplus M_{2})$.
\end{proof}

Now let $P$ be a finitely generated projective $R$-module. We denote by $Um(P)$ the set of epimorphisms $P \rightarrow R$. The group $Aut(P)$ of automorphisms of $P$ then acts on the right on $Um(P)$; consequently, the same holds for any subgroup of $Aut (P)$. In particular, it holds for the subgroup $SL (P)$ of automorphisms of determinant $1$ and, if we fix a decomposition $P \cong \bigoplus_{i=1}^{n} P_{i}$, for the group $E (P) = E(P_1,...,P_n)$ as well.\\
An element $p \in P$ is called unimodular if there is an $a \in Um (P)$ such that $a (p) = 1$; this means that the morphism $R \rightarrow P, 1 \mapsto p$ defines a section for the epimorphism $a$. We denote by $Unim.El. (P)$ the set of unimodular elements of $P$. Note that the group $Aut (P)$ and hence also $SL (P)$ and $E (P)$ act on the left on $P$; these actions restrict to actions on $Unim.El. (P)$.\\
The canonical isomorphism $can: P \rightarrow P^{\vee\vee}$ identifies the set of unimodular elements $Unim.El. (P)$ of $P$ with the set $Um (P^{\vee})$ of epimorphisms $P^{\vee} \rightarrow R$, i.e. an element $p \in P$ is unimodular if and only if $ev_{p}: P^{\vee} \rightarrow R$ is an epimorphism. Furthermore, if $p$ and $q$ are unimodular elements of $P$ and $\varphi \in Aut (P)$ with $\varphi (p) = q$, then $ev_{p} {\varphi}^{\vee} = ev_{q} : P^{\vee} \rightarrow R$.\\
We therefore obtain a well-defined map

\begin{center}
$Unim.El. (P)/Aut (P) \rightarrow Um (P^{\vee})/Aut (P^{\vee})$.
\end{center}

Let us show that this map is actually a bijection. Since the map is automatically surjective, it only remains to show that it is injective. So let $\psi \in Aut (P^{\vee})$ such that $ev_{p} \psi = ev_{q}$. One can easily check that the map $Aut (P) \rightarrow Aut (P^{\vee})$, $\varphi \mapsto {\varphi}^{\vee}$, is bijective; hence $\psi = {\varphi}^{\vee}$ for some $\varphi \in Aut (P)$. Thus, we obtain $ev_{q} = ev_{p} {\varphi}^{\vee}= ev_{\varphi (p)}$ and therefore $\varphi (p) = q$, because $can: P \rightarrow P^{\vee\vee}$ is injective. Altogether, we obtain a bijection
 
\begin{center}
$Unim.El. (P)/Aut (P) \xrightarrow{\cong} Um (P^{\vee})/Aut (P^{\vee})$.
\end{center}
 
In particular, if $P \cong P^{\vee}$, then $Unim.El. (P)/Aut (P) \cong Um (P)/Aut (P)$.\\

We introduce some notation. Let $P_0$ be a finitely generated projective $R$-module. For any $n \geq 3$, let $P_n = P_0 \oplus R e_3 \oplus ... \oplus R e_n$ be the direct sum of $P_0$ and free $R$-modules $R e_i$, $3 \leq i \leq n$, of rank $1$ with explicit generators $e_i$. We denote by $\pi_{k, n}: P_n \rightarrow R$ the projections onto the free direct summands of rank $1$ with index $k = 3, ...,n$. For any non-degenerate alternating form $\chi$ on $P_{2n}$, $n \geq 2$, we define ${Sp} (\chi) = \{\varphi \in Aut (P_{2n}) | {\varphi}^{t} \chi \varphi = \chi \}$.\\
For $n \geq 3$, we have embeddings $Aut (P_{n}) \rightarrow Aut (P_{n+1})$ and $E (P_{n}) \rightarrow E (P_{n+1})$. We denote by $Aut_{\infty} (P_0)$ (resp. $E_{\infty} (P_0)$) the direct limits of the groups $Aut (P_{n})$ (resp. $E (P_{n})$) via these embeddings.\\
In the following lemmas, we denote by $\psi_{2}$ the alternating form on $R^{2}$ given by the matrix
 
\begin{center}
$
\begin{pmatrix}
0 & 1 \\
- 1 & 0
\end{pmatrix}
$.
\end{center}

Thus, for any non-degenerate alternating form $\chi$ on $P_{2n}$ for some $n \geq 2$, we obtain a non-degenerate alternating form on $P_{2n+2}$ given by the orthogonal sum $\chi \perp \psi_{2}$.\\
With this notation in mind, we may now state and prove a few lemmas which provide the technical groundwork in the proofs of the main results in this paper:
 
\begin{Lem}\label{L2.7}
Let $\chi$ be a non-degenerate alternating form on $P_{2n}$ for some $n \geq 2$. Let $p \in P_{2n-1}$ and $a: P_{2n-1} \rightarrow R$. Then there are $\varphi, \psi \in Aut (P_{2n-1})$ such that
\begin{itemize}
\item the morphism $(\varphi \oplus 1) (id_{P_{2n}} + p \pi_{2n, 2n} )$ is an element of $E (P_{2n}) \cap {Sp} (\chi)$ and
\item the morphism $(\psi \oplus 1) (id_{P_{2n}} + a e_{2n})$ is an element of $E (P_{2n}) \cap {Sp} (\chi)$.
\end{itemize}
\end{Lem}

\begin{proof}
We let $\Phi: P_{2n} \rightarrow P_{2n}^{\vee}$ be the alternating isomorphism induced by $\chi$ and ${\Phi}^{-1}$ be its inverse.\\
For the first part, we introduce the following homomorphisms: Let $d$ be the morphism $R \rightarrow P_{2n-1}$ which sends $1$ to ${\Phi}^{-1} (\pi_{2n, 2n})$ (note that because of $\pi_{2n, 2n} ({\Phi}^{-1} (\pi_{2n, 2n})) = \chi ({\Phi}^{-1} (\pi_{2n, 2n}), {\Phi}^{-1} (\pi_{2n, 2n})) = 0$ it can be considered an element of $P_{2n-1}$). Furthermore, let $\nu = \chi (p, -) : P_{2n-1} \rightarrow R$. We observe that $\nu d = 0$. By Lemma \ref{L2.6}, the morphism $\varphi = id_{P_{2n-1}} - d \nu$ is an automorphism and $\varphi \oplus 1$ is an elementary automorphism. In particular, $(\varphi \oplus 1) (id_{P_{2n}} + p \pi_{2n, 2n})$ is an elementary automorphism. In light of the proof of \cite[Lemma 5.4]{SV} and Lemma \ref{L2.2}, one can check locally that it also lies in ${Sp} (\chi)$.\\
For the second part, we introduce the following homomorphisms: We will denote $c = \chi (-, e_{2n}): P_{2n-1} \rightarrow R$. Furthermore, we let $a \oplus 0: P_{2n} \rightarrow R$ be the extension of $a$ to $P_{2n}$ which sends $e_{2n}$ to $0$; then we denote by $\vartheta$ the homomorphism $R \rightarrow P_{2n-1}$ which sends $1$ to $\pi {\Phi}^{-1} (a \oplus 0)$, where $\pi: P_{2n} \rightarrow P_{2n-1}$ is the projection. Note that $c \vartheta = 0$. Again by Lemma \ref{L2.6}, the morphism $\psi = id_{P_{2n-1}} - \vartheta c$ is an automorphism and $\psi \oplus 1$ is an elementary automorphism. In particular, $(\psi \oplus 1) (id_{P_{2n}} + a e_{2n})$ is an elementary automorphism as well. Again, in light of the proof of \cite[Lemma 5.4]{SV} and Lemma \ref{L2.2}, one can check locally that it also lies in ${Sp} (\chi)$.
\end{proof}

\begin{Lem}\label{L2.8}
Let $\chi$ be a non-degenerate alternating form on the module $P_{2n}$ for some $n \geq 2$. Then $E (P_{2n}) e_{2n} = (E (P_{2n}) \cap {Sp} (\chi)) e_{2n}$.
\end{Lem}

\begin{proof}
Let $p \in E (P_{2n}) e_{2n}$. By Corollary \ref{C2.4}, the group $E (P_{2n})$ is generated by automorphisms of the form $id_{P_{2n}} + s$, where $s$ is a morphism $P_{2n-1} \rightarrow Re_{2n}$ or $Re_{2n} \rightarrow P_{2n-1}$. Hence we can write $(\alpha_{1} ... \alpha_{r}) (p) = e_{2n}$, where each $\alpha_{i}$ is one of these generators. We show by induction on $r$ that $p \in (E (P_{2n}) \cap {Sp} (\chi)) e_{2n}$. If $r=0$, there is nothing to show. So let $r \geq 1$. Lemma \ref{L2.7} shows that there is $\gamma \in Aut (P_{2n-1})$ such that $(\gamma \oplus 1) \alpha_{r}$ lies in $E (P_{2n}) \cap {Sp} (\chi)$. We set $\beta_{i} = (\gamma \oplus 1) \alpha_{i} ({\gamma}^{-1} \oplus 1)$ for each $i < r$. Each of the $\beta_{i}$ lies in $E (P_{2n})$ and is again one of the generators of $E (P_{2n})$ given above. By construction, we furthermore have $(\beta_{1} ... \beta_{r-1} (\gamma \oplus 1) \alpha_{r}) (p) = e_{2n}$. This enables us to conclude by induction.
\end{proof}
 
\begin{Lem}\label{L2.9}
Let $\chi_{1}$ and $\chi_{2}$ be non-degenerate alternating forms on $P_{2n}$ such that ${\varphi}^{t} (\chi_{1} \perp \psi_{2}) \varphi = \chi_{2} \perp \psi_{2}$ for some $\varphi \in E_{\infty} (P_{0}) \cap Aut (P_{2n+2})$. Now let $\chi = \chi_{1} \perp \psi_{2}$. If $(E_{\infty} (P_{0}) \cap Aut (P_{2n+2})) e_{2n+2} = (E_{\infty} (P_{0}) \cap {Sp} (\chi)) e_{2n+2}$ holds, then one has ${\psi}^{t} \chi_{2} \psi = \chi_{1}$ for some $\psi \in E_{\infty} (P_{0}) \cap Aut (P_{2n})$.
\end{Lem}

\begin{proof}
Let ${\psi''} e_{2n+2} = \varphi e_{2n+2}$ for some ${\psi''} \in E_{\infty} (P_{0}) \cap {Sp} (\chi)$. Then we simply define ${\psi'} = {(\psi'')}^{-1} \varphi$. Since ${(\psi')}^{t} (\chi_{1} \perp \psi_{2}) {\psi'} = \chi_{2} \perp \psi_{2}$, the composite $\psi: P_{2n} \xrightarrow{\psi'} P_{2n+2} \rightarrow P_{2n}$ and $\psi'$ satisfy the following conditions:

\begin{itemize}
\item ${\psi}^{t} \chi_{1} \psi = \chi_{2}$;
\item ${\psi'} (e_{2n+2}) = e_{2n+2}$;
\item $\pi_{2n+1, 2n+2} {\psi'} = \pi_{2n+1, 2n+2}$.
\end{itemize}
 
The last two conditions imply that $\psi$ equals ${\psi'}$ up to elementary automorphisms and $\psi \in E_{\infty} (P_0) \cap Aut (P_{2n})$, which finishes the proof.
\end{proof}

\begin{Lem}\label{L2.10}
Assume that $\pi_{2n+1, 2n+1} (E_{\infty} (P_0) \cap Aut (P_{2n+1})) = Um (P_{2n+1})$ holds for some $n \in \mathbb{N}$. Then for any non-degenerate alternating form $\chi$ on $P_{2n+2}$ there is an automorphism $\varphi \in E_{\infty} (P_0) \cap Aut (P_{2n+2})$ such that ${\varphi}^{t} \chi \varphi = {\psi} \perp \psi_{2}$ for some non-degenerate alternating form $\psi$ on $P_{2n}$.
\end{Lem}

\begin{proof}
Let $d = \chi(-, e_{2n+2}): P_{2n+1} \rightarrow R$. Since $d$ can be locally checked to be an epimorphism, there is an automorphism ${\varphi'} \in E_{\infty} (P_0) \cap Aut (P_{2n+1})$ such that ${d} {\varphi'} = \pi_{2n+1, 2n+1}$. Then the alternating form ${\chi'} = ({\varphi'}^{t} \oplus 1) \chi  ({\varphi'} \oplus 1)$ satisfies that ${\chi'} (-, e_{2n+2}) : P_{2n+1} \rightarrow R$ is just $\pi_{2n+1, 2n+1}$. Now we simply define $c = {\chi'} (-, e_{2n+1}): P_{2n+1} \rightarrow R$ and let ${\varphi}_{c} = id_{P_{2n+2}} + c e_{2n+2}$ be the elementary automorphism on $P_{2n+2}$ induced by $c$; then ${{\varphi}_{c}}^{t} {\chi'} {\varphi}_{c} = \psi \perp \psi_{2}$ for some non-degenerate alternating form $\psi$ on $P_{2n}$, as desired.
\end{proof}

\begin{Lem}\label{L2.11}
Let $P_0$ be a finitely generated projective $R$-module of rank $2$. Then we have $E (P_{0} \oplus R) \subset SL (P_{0} \oplus R)$. Furthermore, if $\varphi \in SL (P_0 \oplus R)$, then the induced morphism $\varphi_{*}: \det (P_0 \oplus R) \rightarrow \det (P_0 \oplus R)$ is the identity on $\det (P_0 \oplus R)$.
\end{Lem}
 
\begin{proof}
Use that these properties are local and check them when $R$ is local.
\end{proof}

\subsection{The local-global principle for transvection groups}\label{The local-global principle for transvection groups}\label{2.3}

We will now briefly review the local-global principle for transvection groups proven in \cite{BBR} and use it in order to deduce stability results for stably elementary automorphisms of $P_{0} \oplus R^{2}$. For this, we only have to assume that $R$ is an arbitrary commutative ring with unit.\\
First of all, let $P$ be a finitely generated projective $R$-module and $q \in P$, $\varphi \in P^{\vee}$ such that $\varphi (q) = 0$. This data naturally determines a homomorphism $\varphi_{q}: P \rightarrow P$ by $\varphi_{q} (p) = \varphi (p) q$ for all $p \in P$. An automorphism of the form $id_{P} + \varphi_{q}$ is called a transvection if either $q \in Unim.El. (P)$ or $\varphi \in Um (P)$. We denote by $T(P)$ the subgroup of $Aut (P)$ generated by transvections.\\
Now let $Q = P \oplus R$ be a direct sum of a finitely generated projective $R$-module $P$ of rank $\geq 2$ and the free $R$-module of rank $1$. Then the elementary automorphisms of $P \oplus R$ are easily seen to be transvections and are also called elementary transvections. Consequently, we have $E (Q) \subset T (Q) \subset Aut (Q)$.\\
In the theorem stated below, we denote by $R [X]$ the polynomial ring in one variable over $R$ and let $Q [X] = Q \otimes_{R} R[X]$. The evaluation homomorphisms $ev_{0}, ev_{1}: R[X] \rightarrow R$ induce maps $Aut (Q[X]) \rightarrow Aut (Q)$. If $\alpha (X) \in Aut (Q[X])$, then we denote its images under these maps by $\alpha (0)$ and $\alpha (1)$ respectively. Similarly, the localization homomorphism $R \rightarrow R_{\mathfrak{m}}$ at any maximal ideal $\mathfrak{m}$ of $R$ induces a map $Aut (Q[X]) \rightarrow Aut (Q_{\mathfrak{m}} [X])$, where $Q_{\mathfrak{m}} [X] = Q[X] \otimes_{R[X]} R_{\mathfrak{m}}[X]$; if $\alpha (X) \in Aut (Q[X])$, we denote its image under this map by $\alpha_{\mathfrak{m}} (X)$.\\
We will use the following result proven by Bak, Basu and Rao (see \cite[Theorems 3.6 and 3.10]{BBR}):

\begin{Thm}\label{T2.12}
The inclusion $E (Q) \subset T (Q)$ is an equality. If $\alpha (X) \in Aut (Q[X])$ satisfies $\alpha (0) = id_{Q} \in Aut (Q)$ and $\alpha_{\mathfrak{m}} (X) \in E (Q_{\mathfrak{m}} [X])$ for all maximal ideals $\mathfrak{m}$ of $R$, then $\alpha (X) \in E (Q[X])$.
\end{Thm}

In order to prove the desired stability results, we introduce the following property: Let $\mathfrak{C}$ be either the class of Noetherian rings or the class of affine $k$-algebras over a fixed field $k$. Furthermore, let $d \geq 1$ be an integer and $m \in \mathbb{N}$. We say that $\mathfrak{C}$ has the property $\mathcal{P} (d,m)$ if for $R$ in $\mathfrak{C}$ of dimension $d$ and for any finitely generated projective $R$-module $P$ of rank $m$ the group $SL(P \oplus R^{n})$ acts transitively on $Um (P \oplus R^{n})$ for all $n \geq 2$.\\
If $k$ is a field, we simply say that $k$ has the property $\mathcal{P} (d,m)$ if the class of affine $k$-algebras has the property $\mathcal{P} (d,m)$.\\
Of course, if the class of Noetherian rings has the property $\mathcal{P} (d,m)$, then the same holds for every field. The class of Noetherian rings has the property $\mathcal{P} (d,m)$ for $m \geq d$. Furthermore, it follows from \cite{B} that any perfect field $k$ of cohomological dimension $\leq 1$ satisfies property $\mathcal{P} (d,d-1)$ if $d! \in k^{\times}$.\\
In the remainder of this section, we will denote by $\pi$ the canonical projection $P \oplus R^{n} \rightarrow R$ onto the "last" free direct summand of $R^{n}$.

\begin{Lem}\label{L2.13}
Let $\mathfrak{C}$ be the class of Noetherian rings or affine $k$-algebras over a fixed field $k$. Assume that $\mathfrak{C}$ has the property $\mathcal{P} (d,m)$. Let $R$ be a $d$-dimensional ring in $\mathfrak{C}$, $P$ a projective $R$-module of rank $m$ and $a \in Um (P \oplus R^{n})$ for some $n \geq 2$. Moreover, assume that there is an element $t \in R$ and a homomorphism $w: P \oplus R^{n} \rightarrow R$ such that $a - \pi = t w$. Then there is $\varphi \in SL (P \oplus R^{n})$ such that $a = \pi \varphi$ and $\varphi (x) \equiv id_{P \oplus R^{n}} (x)$ modulo $(t)$ for all $x$.
\end{Lem}

\begin{proof}
We set $B = R[X]/\langle X^{2}-tX\rangle$. By assumption, we have $a = \pi + t w$. We lift it to $a (X) = \pi + X w: (P \oplus R^{n})\otimes_{R} B\rightarrow B$, which can be checked to be an epimorphism (as in the proof of \cite[Proposition 3.3]{RvdK}). Therefore we have $a (X) \in Um ((P \oplus R^{n})\otimes_{R} B)$. Since $B$ still is a ring in $\mathfrak{C}$ of dimension $d$, property $\mathcal{P} (d,m)$ now gives an element $\varphi (X) \in SL ((P \oplus R^{n})\otimes_{R} B)$ with $a (X) = \pi \varphi (X)$. Then $\varphi = {\varphi (0)}^{-1} \varphi (t)$ is the desired automorphism.
\end{proof}

For any $n \geq 2$, we say that two automorphisms $\varphi, \psi \in SL (P \oplus R^{n})$ are isotopic if there is an automorphism $\tau (X)$ of $(P \oplus R^{n})\otimes_{R} R[X]$ with determinant $1$ such that $\tau (0) = \varphi$ and $\tau (1) = \psi$.

\begin{Thm}\label{T2.14}
Let $\mathfrak{C}$ be the class of Noetherian rings or affine $k$-algebras over a fixed field $k$. Assume that $\mathfrak{C}$ has the property $\mathcal{P} (d+1,m+1)$. Let $R$ be a $d$-dimensional ring in $\mathfrak{C}$, $P$ a projective $R$-module of rank $m$ and $\sigma \in Aut (P \oplus R^{n})$ for some $n \geq 2$. Assume that $\sigma \oplus 1 \in E (P \oplus R^{n+1})$. Then $\sigma$ is isotopic to $id_{P \oplus R^{n}}$.
\end{Thm}

\begin{proof}
Since ${\sigma \oplus 1} \in E (P \oplus R^{n+1})$, it is clear that there is a natural isotopy $\tau (X) \in E ((P \oplus R^{n+1})\otimes_{R} R[X])$ with $\tau (0) = id_{P \oplus R^{n+1}}$ and $\tau (1) = \sigma \oplus 1$. Now apply the previous lemma to $R [X]$, $X^{2} - X$ and $a = \pi \tau (X)$ in order to obtain an automorphism $\chi (X) \in SL ((P \oplus R^{n+1})\otimes_{R} R[X])$ with $\pi \chi (X) = a$ such that ${\chi(X)} (x) \equiv x$ modulo $\langle X^{2}-X\rangle$. Thus, $\pi \tau (X) {\chi(X)}^{-1} = \pi$. Therefore $\tau (X) {\chi (X)}^{-1}$ equals $\rho (X) \oplus 1$ for some $\rho (X) \in SL ((P \oplus R^{n})\otimes_{R} R[X])$ up to elementary automorphisms. But then $\rho (X)$ is an isotopy from $id_{P \oplus R^{n}}$ to $\sigma$.
\end{proof}

We can now use Theorem \ref{T2.14} in order to deduce the following stability results:

\begin{Thm}\label{T2.15}
With the notation of Section \ref{2.2}, we further assume that $P_{0}$ has rank $2$. If $R$ is a regular Noetherian ring of dimension $2$, then there is an equality $E_{\infty} (P_{0}) \cap Aut (P_{4}) = E (P_{4})$.
\end{Thm}

\begin{proof}
If $\sigma \in SL (P_{4})$ is stably elementary, then $\sigma \in E (P_{n+1})$ for some $n \geq 4$. We can now apply Theorem \ref{T2.14} to $P = P_{0}$ and deduce that there is an isotopy $\rho (X) \in SL (P_{n} [X])$ from $id_{P_{n}}$ to $\sigma$.\\
But since $R$ is regular, we know that $\rho_{\mathfrak{m}} (X)$ is stably elementary (for any maximal ideal $\mathfrak{m}$ of $R$); in fact, we can deduce that $\rho_{\mathfrak{m}} (X)$ is elementary: Since $\dim (R) = 2$, the spectrum of the $3$-dimensional ring $R_{\mathfrak{m}}[X]$ is the union of a finite number of subspaces of dimension $\leq 2$ (cf. the last paragraph of \cite[Section 1.1]{R}). Hence it follows from \cite[Lemma 7.5]{SV} that the stable rank of $R_{\mathfrak{m}}[X]$ is at most $3$; in particular, $SL_{n} (R_{\mathfrak{m}}[X]) \cap E (R_{\mathfrak{m}}[X]) = E_{n} (R_{\mathfrak{m}}[X])$ and $\rho_{\mathfrak{m}} (X)$ is elementary.\\
Then Theorem \ref{T2.12} implies that $\rho (X) \in E (P_{n}[X])$ and hence $\sigma = \rho (1) \in E (P_{n})$. The theorem now follows by inductively repeating this argument and deducing that $\sigma \in E (P_{4})$.
\end{proof}

\begin{Thm}\label{T2.16}
With the notation of Section \ref{2.2}, we further assume that $P_{0}$ has rank $2$. Let $k$ be a perfect field with $\mathcal{P} (4,3)$. If $R$ is a regular affine $k$-algebra of dimension $3$, then $E_{\infty} (P_{0}) \cap Aut (P_{4}) = E (P_{4})$.
\end{Thm}

\begin{proof}
By a famous theorem of Vorst (cf. \cite{V}), we know that there is an equality $SL_{N} (R_{\mathfrak{p}} [X]) = E_{N} (R_{\mathfrak{p}} [X])$ for any prime $\mathfrak{p}$ of $R$ and $N \geq 4$. We can thus argue as in the proof of Theorem \ref{T2.15}.
\end{proof}

\begin{Thm}\label{T2.17}
With the notation of Section \ref{2.2}, we further assume that $P_{0}$ has rank $2$. Let $k$ be a perfect field with $\mathcal{P} (5,3)$. If $R$ is a regular affine $k$-algebra of dimension $4$, then $E_{\infty} (P_{0}) \cap Aut (P_{4}) = E (P_{4})$.
\end{Thm}

\begin{proof}
We can argue as in the proof of Theorem \ref{T2.16}.
\end{proof}

\subsection{Classification of stably isomorphic projective modules}\label{Classification of stably isomorphic projective modules}\label{2.4}

We consider the map

\begin{center}
$\phi_{n} : \mathcal{V}_{n} (R) \rightarrow \mathcal{V}_{n+1} (R), [P] \mapsto [P \oplus R]$,
\end{center}
 
from isomorphism classes of rank $n$ projective modules to rank $n+1$ projective modules and fix a projective module $P \oplus R$ representing an element of ${\mathcal{V}}_{n+1} (R)$ in the image of this map. An element $[P']$ of ${\mathcal{V}}_{n} (R)$ lies in the fiber over $[P \oplus R]$ if and only if there is an isomorphism $i: P' \oplus R \xrightarrow{\cong} P \oplus R$. Any such isomorphism yields an element of $Um (P \oplus R)$ given by the composite 

\begin{center}
$a (i): P \oplus R \xrightarrow{{i}^{-1}} P' \oplus R \xrightarrow{\pi_R} R$.
\end{center}

Note that if one chooses another module $P''$ representing the isomorphism class of $P'$ and any isomorphism $j: P'' \oplus R \xrightarrow{\cong} P \oplus R$, the resulting element $a (j)$ of $Um (P \oplus R)$ still lies in the same orbit of $Um (P \oplus R)/Aut (P \oplus R)$:
For if we choose an isomorphism $k: P' \xrightarrow{\cong} P''$, then we have an equality

\begin{center}
$a (i) = a (j) \circ ({j} {(k \oplus id_{R})} {i}^{-1})$.
\end{center}

Thus, we obtain a well-defined map

\begin{center}
$\phi_{n}^{-1} ([P \oplus R]) \rightarrow Um (P \oplus R)/Aut (P \oplus R)$.
\end{center}

Conversely, any element $a$ of $Um (P \oplus R)$ gives an element of ${\mathcal{V}}_{n} (R)$ lying over $[P \oplus R]$, namely $[P'] = [\ker (a)]$. Note that the kernels of two epimorphisms $P \oplus R \rightarrow R$ are isomorphic if these epimorphisms are in the same orbit in $Um (P \oplus R)/Aut (P \oplus R)$. Thus, we also obtain a well-defined map

\begin{center}
$Um (P \oplus R)/Aut(P \oplus R) \rightarrow \phi_{n}^{-1} ([P \oplus R])$.
\end{center}

One can easily check that the maps $\phi_{n}^{-1} ([P \oplus R]) \rightarrow Um (P \oplus R)/Aut (P \oplus R)$ and $Um (P \oplus R)/Aut(P \oplus R) \rightarrow \phi_{n}^{-1} ([P \oplus R])$ are inverse to each other. Note that $[P]$ corresponds to the class represented by the canonical projection $\pi_{R} : P \oplus R \rightarrow R$ under these bijections. In conclusion, we have a pointed bijection between the sets $Um (P \oplus R)/Aut(P \oplus R)$ and $\phi_{n}^{-1} ([P \oplus R])$ equipped with $[\pi_{R}]$ and $[P]$ as their respective basepoints. Moreover, we also obtain a (pointed) surjection $Um (P \oplus R)/E(P \oplus R) \rightarrow \phi_{n}^{-1} ([P \oplus R])$.

\section{The elementary symplectic Witt group}\label{The elementary symplectic Witt group}\label{3}

In this section, we briefly recall the definition of the so-called elementary symplectic Witt group $W_{E} (R)$. Primarily, it appears as the kernel of a homomorphism $W'_{E} (R) \rightarrow R^{\times}$ induced by the Pfaffian of alternating invertible matrices. As we will discuss, the group $W'_{E} (R)$ itself can be identified with a group denoted $V (R)$ and with $GW_{1}^{3} (R)$, a higher Grothendieck-Witt group of $R$. We will also prove some lemmas on the group $V (R)$, which will be used to prove the main results of this paper. Furthermore, we introduce a canonical $R^{\times}$-action on $V (R)$ and identify this action with an action of $R^{\times}$ on $GW_{1}^{3} (R)$ coming from the multiplicative structure of higher Grothendieck-Witt groups.
 
\subsection{The group $W'_E (R)$}\label{The group $W'_E (R)$}\label{3.1}

Let R be a commutative ring. For any $n \in \mathbb{N}$, we denote by $A_{2n} (R)$ the set of alternating invertible matrices of rank $2n$. We inductively define an element $\psi_{2n} \in  A_{2n} (R)$ by setting
 
\begin{center}
$\psi_2 =
\begin{pmatrix}
0 & 1 \\
- 1 & 0
\end{pmatrix}
$
\end{center}
 
\noindent and $\psi_{2n+2} = \psi_{2n} \perp \psi_2$. For any $m < n$, there is an embedding of $A_{2m} (R)$ into $A_{2n} (R)$ given by $M \mapsto M \perp \psi_{2n-2m}$. We denote by $A (R)$ the direct limit of the sets $A_{2n} (R)$ under these embeddings. Two alternating invertible matrices $M \in A_{2m} (R)$ and $N \in A_{2n} (R)$ are called equivalent, $M \sim N$, if there is an integer $s \in \mathbb{N}$ and a matrix $E \in E_{2n+2m+2s}$ such that

\begin{center}
$M \perp \psi_{2n+2s} = E^{t} (N \perp \psi_{2m+2s}) E$.
\end{center}

The set of equivalence classes $A(R)/{\sim}$ is denoted $W'_E (R)$. Since
 
\begin{center}
$
\begin{pmatrix}
0 & id_{s} \\
id_{r} & 0
\end{pmatrix}
\in E_{r+s} (R)$
\end{center}
 
for even $rs$, it follows that the orthogonal sum equips $W'_E (R)$ with the structure of an abelian monoid. As it is shown in \cite{SV}, this abelian monoid is actually an abelian group. An inverse for an element of $W'_E (R)$ represented by a matrix $N \in A_{2n} (R)$ is given by the element represented by the matrix $\sigma_{2n} N^{-1} \sigma_{2n}$, where the matrices $\sigma_{2n}$ are inductively defined by
 
\begin{center}
$\sigma_2 =
\begin{pmatrix}
0 & 1 \\
1 & 0
\end{pmatrix}
$
\end{center}
 
\noindent and $\sigma_{2n+2} = \sigma_{2n} \perp \sigma_2$.
Now recall that one can assign to any alternating invertible matrix $M$ an element $\mathit{Pf}(M)$ of $R^{\times}$ called the Pfaffian of $M$. The Pfaffian satisfies the following formulas:
 
\begin{itemize}
\item $\mathit{Pf}(M \perp N) = \mathit{Pf}(M) \mathit{Pf}(N)$ for all $M \in A_{2m} (R)$ and $N \in A_{2n} (R)$
\item $\mathit{Pf}(G^{t} N G) = \det (G) \mathit{Pf}(N)$ for all $G \in GL_{2n} (R)$ and $N \in A_{2n} (R)$
\item ${\mathit{Pf} (N)}^{2} = \det (N)$ for all $N \in A_{2n} (R)$
\item $\mathit{Pf} (\psi_{2n}) = 1$ for all $n \in \mathbb{N}$
\end{itemize}

\noindent Therefore the Pfaffian determines a group homomorphism $\mathit{Pf}: W'_E (R) \rightarrow R^{\times}$; its kernel is denoted $W_E (R)$ and is called the elementary symplectic Witt group of $R$. Note that the homomorphism $\mathit{Pf}: W'_E (R) \rightarrow R^{\times}$ is split by the homomorphism $R^{\times} \rightarrow W'_E (R)$ which assigns to any $t \in R^{\times}$ the class in $W'_E (R)$ represented by the matrix

\begin{center}
$\begin{pmatrix}
0 & t \\
-t & 0
\end{pmatrix}
$.
\end{center}

Hence $W'_E (R) \cong W_E (R) \oplus R^{\times}$.

\subsection{The group V(R)}\label{The group V(R)}\label{3.2}

Again, let $R$ be a commutative ring. Consider the set of triples $(P, g, f)$, where $P$ is a finitely generated projective $R$-module and $f,g$ are alternating isomorphisms on $P$. Two such triples $(P, f_0, f_1)$ and $(P', f'_0, f'_1)$ are called isometric if there is an isomorphism $h: P \rightarrow P^{'}$ such that $f_{i} = h^{\vee} f'_{i} h$ for $i=0,1$. We denote by $[P, g, f]$ the isometry class of the triple $(P, g, f)$.\\
Let $V (R)$ be the quotient of the free abelian group on isometry classes of triples as above by the subgroup generated by the relations
 
\begin{itemize}
\item $[P \oplus P', g \perp g', f \perp f'] = [P, g, f] + [P', g', f']$ for alternating isomorphisms $f,g$ on $P$ and $f',g'$ on $P'$,
\item $[P, f_0, f_1] + [P, f_1, f_2] = [P, f_0, f_2]$ for alternating isomorphisms $f_{0}, f_{1}, f_{2}$ on $P$.
\end{itemize}

Note that these relations yield the useful identities
 
\begin{itemize}
\item $[P, f, f] = 0$ in $V (R)$ for any alternating isomorphism $f$ on $P$,
\item $[P, g, f] = - [P, f, g]$ in $V (R)$ for alternating isomorphisms $f,g$ on $P$,
\item $[P, g, {\beta}^{\vee} {\alpha}^{\vee} f \alpha \beta] = [P, f, {\alpha}^{\vee} f \alpha] + [P, g, {\beta}^{\vee} f \beta]$ in $V (R)$ for all automorphisms $\alpha, \beta$ of $P$ and alternating isomorphisms $f,g$ on $P$.
\end{itemize}

We may also restrict this construction to free $R$-modules of finite rank. The corresponding group will be denoted $V_{\mathit{free}}(R)$. Note that there is an obvious group homomorphism $V_{\mathit{free}}(R) \rightarrow V (R)$.

This homomorphism can be seen to be an isomorphism as follows: For any finitely generated projective $R$-module $P$, we call

\begin{center}
$H_P =
\begin{pmatrix}
0 & id_{P^{\vee}} \\
-can & 0
\end{pmatrix}
: P \oplus P^{\vee} \rightarrow P^{\vee} \oplus P^{\vee\vee}$
\end{center}

the hyperbolic isomorphism on $P$.
 
Now let $(P, g, f)$ be a triple as above. Since $P$ is a finitely generated projective $R$-module, there is another $R$-module $Q$ such that $P \oplus Q \cong R^{n}$ for some $n \in \mathbb{N}$. In particular, $P \oplus P^{\vee} \oplus Q \oplus Q^{\vee}$ is free of rank $2n$. Therefore the triple
 
\begin{center}
$(P \oplus P^{\vee} \oplus Q \oplus Q^{\vee}, g \perp can\; g^{-1} \perp H_Q, f \perp can\; g^{-1} \perp H_Q)$
\end{center}
 
represents an element of $V_{\mathit{free}}(R)$; this element is independent of the choice of $Q$. It follows that the assignment
 
\begin{center}
$(P, g, f) \mapsto (P \oplus P^{\vee} \oplus Q \oplus Q^{\vee}, g \perp can\; g^{-1} \perp H_Q, f \perp can\; g^{-1} \perp H_Q)$
\end{center}
 
induces a well-defined group homomorphism
 
\begin{center}
$V (R) \rightarrow V_{\mathit{free}}(R)$.
\end{center}
 
Since
 
\begin{center}
$[P, g, f] = [P \oplus P^{\vee} \oplus Q \oplus Q^{\vee}, g \perp can\; g^{-1} \perp H_Q, f \perp can\; g^{-1} \perp H_Q]$
\end{center}

in $V (R)$, this homomorphism is actually inverse to the canonical morphism $V_{\mathit{free}}(R) \rightarrow V (R)$. Thus, $V_{\mathit{free}}(R) \cong V (R)$.\\

In order to discuss the identification of $V (R)$ with the group $W'_{E} (R)$ described in the previous section, we first need to prove Lemma \ref{L3.1} and Corollaries \ref{C3.2} and \ref{C3.3} below. They will also be used in the proofs of the main results of this paper.

\begin{Lem}\label{L3.1}
Let $P = \bigoplus_{i=1}^{n} P_{i}$ be a finitely generated projective module and $f_{i}$ alternating isomorphisms on $P_{i}$, $i=1,...,n$. Let $f = f_1 \perp ... \perp f_n$. Then $[P, f, {\varphi}^{\vee} f {\varphi}] = 0$ in $V (R)$ for any element ${\varphi}$ of the commutator subgroup of $Aut (P)$. In particular, the same holds for every element of $E (P)$ with respect to the given decomposition.
\end{Lem}

\begin{proof}
By the third of the useful identities listed above, we have
 
\begin{center}
$[P, f, {\varphi}_{2}^{\vee} {\varphi}_{1}^{\vee} f {\varphi}_{1} {\varphi}_{2}] = [P, f, {\varphi}_{1}^{\vee} f {\varphi}_{1}] + [P, f, {\varphi}_{2}^{\vee} f {\varphi}_{2}]$.
\end{center}
 
Therefore we only have to prove the first statement for commutators. Now if $\varphi = {\varphi}_{1} {\varphi}_{2} {\varphi}_{1}^{-1} {\varphi}_{2}^{-1}$ is a commutator, then the formula above yields
 
\begin{center}
$[P, f, {\varphi}^{\vee} f {\varphi}] = [P, f, {\varphi}_{1}^{\vee} f {\varphi}_{1}] + [P, f, {\varphi}_{2}^{\vee} f {\varphi}_{2}] + [P, f, {({\varphi}_{1}^{-1})}^{\vee} f {\varphi}_{1}^{-1}]+ [P, f, {({\varphi}_{2}^{-1})}^{\vee} f {\varphi}_{2}^{-1}] = 0$,
\end{center}
 
which proves first part of the lemma.\\
For the second part, observe that by the formula above we only need to prove the statement for elementary automorphisms. So let $\varphi_{s}$ be the elementary automorphism induced by $s: P_{j} \rightarrow P_{i}$. Since we can add the summand $[P_{i}, f_{i}, f_{i}] = 0$, we may assume that we are in the situation of Corollary \ref{C2.5}. Therefore we may assume that $\varphi_{s}$ is a commutator and the second statement then follows from the first part of the lemma.
\end{proof}

\begin{Kor}\label{C3.2}
Let $P$ be a finitely generated projective $R$-module and $\chi$ an alternating isomorphism on $P$. Then $[P \oplus R^{2n}, \chi \perp \psi_{2n}, {\varphi}^{\vee} (\chi \perp \psi_{2n}) {\varphi}] = 0$ in $V (R)$ for any elementary automorphism ${\varphi}$ of $P \oplus R^{2n}$. In particular, if $f$ is any alternating isomorphism on $P \oplus R^{2n}$, it follows that there is an equality $[P \oplus R^{2n}, \chi \perp \psi_{2n}, {\varphi}^{\vee} f {\varphi}] = [P \oplus R^{2n}, \chi \perp \psi_{2n}, f]$ in $V (R)$.
\end{Kor}
 
\begin{proof}
The first part follows directly from the previous lemma. The second part is then a direct consequence of the second relation given in the definition of the group $V (R)$.
\end{proof}
 
\begin{Kor}\label{C3.3}
For an arbitrary elementary matrix $E \in E_{2n} (R)$, we have $[R^{2n}, \psi_{2n}, E^{t} \psi_{2n} E] = 0$ in $V (R)$. In particular, for any alternating matrix $N \in GL_{2n} (R)$, we have $[R^{2n}, \psi_{2n}, N] = [R^{2n}, \psi_{2n}, E^{t} N E]$ in $V (R)$.
\end{Kor}
 
Using the previous corollary, the group $V_{\mathit{free}}(R)$ can be identified with $W'_E (R)$ as follows: If $M \in A_{2m} (R)$ represents an element of $W'_E (R)$, then we assign to it the class in $V_{\mathit{free}}(R)$ represented by $[R^{2m}, \psi_{2m}, M]$. By Corollary \ref{C3.3}, this assignment descends to a well-defined homomorphism $\nu: W'_E (R) \rightarrow V_{\mathit{free}}(R)$.\\
Now let us describe the inverse $\xi: V_{\mathit{free}}(R) \rightarrow W'_{E} (R)$ to this homomorphism. Let $(L, g, f)$ be a triple with $L$ free and $g,f$ alternating isomorphisms on $L$. We can choose an isomorphism $\alpha: R^{2n} \xrightarrow{\cong} L$ and consider the alternating isomorphism
 
\begin{center}
$({\alpha}^{\vee} f \alpha) \perp \sigma_{2n} {({\alpha}^{\vee} g \alpha)}^{-1} \sigma_{2n}^{\vee}: R^{2n} \oplus {(R^{2n})}^{\vee} \rightarrow {(R^{2n})}^{\vee} \oplus R^{2n}$.
\end{center}
 
With respect to the standard basis of $R^{2n}$ and its dual basis of ${(R^{2n})}^{\vee}$, we may interpret this alternating isomorphism as an element of $A_{4n} (R)$ and then consider its class $\xi ([L, g, f])$ in $W'_E (R)$. In fact, this class is independent of the choice of the isomorphism $\alpha: R^{2n} \xrightarrow{\cong} L$. If $\beta: R^{2n} \xrightarrow{\cong} L$ is another isomorphism, then it suffices to show that the alternating matrix $M$ corresponding to $\alpha^{\vee} f \alpha \perp \beta^{\vee} g \beta$ is equivalent in $W'_E (R)$ to the alternating matrix corresponding to $\beta^{\vee} f \beta \perp \alpha^{\vee} g \alpha$. But there is an isometry $\gamma = ({\alpha}^{-1}\beta) \perp ({\beta}^{-1}\alpha)$ from $\alpha^{\vee} f \alpha \perp \beta^{\vee} g \beta$ to $\beta^{\vee} f \beta \perp \alpha^{\vee} g \alpha$, which is an elementary automorphism by Whitehead's lemma. One then also checks easily that the defining relations of $V_{\mathit{free}}(R)$ are also satisfied by the assignment above. Hence it follows that this assignment induces a well-defined homomorphism $\xi: V_{\mathit{free}}(R) \rightarrow W'_E (R)$. By construction, $\nu$ and $\xi$ are obviously inverse to each other and therefore identify $W'_E (R)$ with $V_{\mathit{free}}(R)$.

In order to conclude this section, let us now describe some group actions on $V (R)$: For any finitely generated projective $R$-module $P$, alternating isomorphism $\chi: P \rightarrow P^{\vee}$ and $u \in R^{\times}$, the morphism $u \cdot \chi: P \rightarrow P^{\vee}$ is again an alternating isomorphism on $P$. Note that $u \cdot \chi$ is canonically isometric to the alternating isomorphism $u \otimes \chi: R \otimes P \xrightarrow{u \otimes \chi} R \otimes P^{\vee} \cong {(R \otimes P)}^{\vee}$ and we therefore have an equality
 
\begin{center}
$[P, u \cdot \chi_{2}, u \cdot \chi_{1}] = [R \otimes P, u \otimes \chi_{2}, u \otimes \chi_{1}]$ in $V (R)$
\end{center}
 
for all $\chi_{1}, \chi_{2}$. One can check easily that the assignment
 
\begin{center}
$(u, (P, \chi_{2}, \chi_{1})) \mapsto (P, u \cdot \chi_{2}, u \cdot \chi_{1})$
\end{center}
 
descends to a well-defined action of $R^{\times}$ on $V (R)$.\\
Now let us assume that $2 \in R^{\times}$ and, furthermore, let $\varphi: Q \rightarrow Q^{\vee}$ be a symmetric isomorphism on a finitely generated projective $R$-module $Q$. Then, for any skew-symmetric isomorphism $\chi$ on a finitely generated projective $R$-module $P$ as above, the homomorphism $\varphi \otimes \chi:  Q \otimes P \rightarrow Q^{\vee} \otimes P^{\vee} \cong {(Q \otimes P)}^{\vee}$ is again a skew-symmetric isomorphism on $Q \otimes P$. One can check easily that the assignment
 
\begin{center}
$((Q, \varphi), (P, \chi_{2}, \chi_{1})) \mapsto (Q \otimes P, \varphi \otimes \chi_{2}, \varphi \otimes \chi_{1})$
\end{center}
 
induces a well-defined action of the Grothendieck-Witt group $GW (R) = GW_{0}^{0} (R)$ of $R$ on $V (R)$.

\subsection{Grothendieck-Witt groups}\label{Grothendieck-Witt groups}\label{3.3}

In this section, we recall some basics about the theory of higher Grothendieck-Witt groups, which are a modern version of Hermitian K-theory. The general references of the modern theory are \cite{MS1}, \cite{MS2} and \cite{MS3}.\\
We assume $R$ to be a ring such that $2 \in R^{\times}$. Then we consider the category $P (R)$ of finitely generated projective $R$-modules and the category $C^{b} (R)$ of bounded complexes of objects in $P (R)$. The category $C^{b} (R)$ inherits a natural structure of an exact category from $P (R)$ by declaring $C'_{\bullet} \rightarrow C_{\bullet} \rightarrow C''_{\bullet}$ exact if and only if $C'_{n} \rightarrow C_{n} \rightarrow C''_{n}$ is exact for all $n$. The duality $Hom_{R} (-, L)$ associated to any line bundle $L$ induces a duality $\#_{L}$ on $C^{b} (R)$ and the identification of a finitely generated projective $R$-module with its double dual induces a natural isomorphism of functors $\varpi_{L}: id \xrightarrow{\sim} \#_{L}\#_{L}$ on $C^{b} (R)$. Moreover, the translation functor $T: C^{b} (R) \rightarrow C^{b} (R)$ yields new dualities $\#_{L}^{j} = T^{j} \#_{L}$ and natural isomorphisms $\varpi_{L}^{j} = {(-1)}^{j (j+1)/2} \varpi_{L}$. We say that a morphism in $C^{b} (R)$ is a weak equivalence if and only if it is a quasi-isomorphism and we denote by $qis$ the class of quasi-isomorphisms. For all $j$, the quadruple $(C^{b} (R), qis, \#_{L}^{j}, \varpi_{L}^{j})$ is an exact category with weak equivalences and strong duality (cf. \cite[\S 2.3]{MS2}).\\
Following \cite{MS2}, one can associate a Grothendieck-Witt space $\mathcal{GW}$ to any exact category with weak equivalences and strong duality. The (higher) Grothendieck-Witt groups are then defined to be its homotopy groups:
 
\begin{Def}
For any $i \geq 0$, let $\mathcal{GW} (C^{b} (R), qis, \#_{L}^{j}, \varpi_{L}^{j})$ be the Grothendieck-Witt space associated to the quadruple $(C^{b} (R), qis, \#_{L}^{j}, \varpi_{L}^{j})$ as above. Then we define ${GW}_{i}^{j} (R, L) = \pi_{i} \mathcal{GW} (C^{b} (R), qis, \#_{L}^{j}, \varpi_{L}^{j})$. If $L = R$, then we set ${GW}_{i}^{j} (R) = {GW}_{i}^{j} (R, L)$.
\end{Def}
 
The groups $GW_{i}^{j} (R, L)$ are $4$-periodic in $j$ and coincide with Hermitian K-theory and U-theory as defined by Karoubi (cf. \cite{MK1} and \cite{MK2}), at least if $2 \in R^{\times}$ (see \cite[Remark 4.13]{MS1} and \cite[Theorems 6.1-2]{MS3}). In particular, we have isomorphisms $K_{i}O (R) = GW_{i}^{0} (R)$, $_{-1}U_{i} (R) = GW_{i}^{1} (R)$, $K_{i}Sp (R) = GW_{i}^{2} (R)$ and $U_{i} (R) = GW_{i}^{3} (R)$.\\
The group of our particular interest is $GW_{1}^{3} (R) = U_{1} (R)$. Indeed, it is argued in \cite{FRS} that there is a natural isomorphism $GW_{1}^{3} (R) \cong V_{\mathit{free}}(R) \cong V (R)$.\\
The Grothendieck-Witt groups defined as above carry a multiplicative structure. Indeed, the tensor product of complexes induces product maps
 
\begin{center}
${GW}_{i}^{j} (R, L_{1}) \times GW_{r}^{s} (R, L_{2}) \rightarrow GW_{i+r}^{j+s} (R, L_{1} \otimes L_{2})$
\end{center}
 
for all $i,j,r,s$ and line bundles $L_{1},L_{2}$ (cf. \cite[\S 9.2]{MS3}). In general, it is (probably) difficult to give explicit descriptions of this multiplicative structure; nevertheless, if we restrict ourselves to smooth algebras over a perfect field $k$ (with $char (k) \neq 2$), then it is known (see \cite[Theorem 3.1]{JH}) that Grothendieck-Witt groups are representable in the (pointed) $\mathbb{A}^{1}$-homotopy category $\mathcal{H}_{\bullet} (k)$ as defined by Morel and Voevodsky. As a matter of fact, if we let $R$ be a smooth affine $k$-algebra over a perfect field $k$ with $char (k) \neq 2$ and $X = Spec (R)$, it is shown that there are spaces $\mathcal{GW}^{j}$ such that
 
\begin{center}
$[S_{s}^{i} \wedge X_{+}, \mathcal{GW}^{j}]_{\mathbb{A}^{1}, \bullet} = GW_{i}^{j} (R)$,
\end{center}
 
i.e. the spaces $\mathcal{GW}^{j}$ represent the higher  Grothendieck-Witt groups. In order to make these spaces more explicit, we consider for all $n \in \mathbb{N}$ the closed embeddings $GL_{n} \rightarrow O_{2n}$ and $GL_{n} \rightarrow Sp_{2n}$ induced by

\begin{center}
$M \mapsto
\begin{pmatrix}
M & 0 \\
0 & {(M^{-1})}^{t}
\end{pmatrix}
$.
\end{center}

These embeddings are compatible with the standard stabilization embeddings $GL_{n} \rightarrow GL_{n+1}$, $O_{2n} \rightarrow O_{2n+2}$ and $Sp_{2n} \rightarrow Sp_{2n+2}$. Taking direct limits over all $n$ with respect to the induced maps $O_{2n}/GL_{n} \rightarrow O_{2n+2}/GL_{n+1}$ and $Sp_{2n}/GL_{n} \rightarrow Sp_{2n+2}/GL_{n+1}$, we obtain spaces $O/GL$ and $Sp/GL$. Similarly, the natural inclusions $Sp_{2n} \rightarrow GL_{2n}$ are compatible with the standard stabilization embeddings and we obtain a space $GL/Sp = colim_{n}\; GL_{2n}/Sp_{2n}$. As proven in \cite[Theorems 8.2 and 8.4]{ST}, there are canonical $\mathbb{A}^{1}$-weak equivalences
 
\begin{center}
\begin{equation*}
   \mathcal{GW}^{j} \cong
   \begin{cases}
     \mathbb{Z} \times OGr & \text{if } j \equiv 0 \text{ mod } 4 \\
     Sp/GL & \text{if } j \equiv 1 \text{ mod } 4 \\
     \mathbb{Z} \times HGr & \text{if } j \equiv 2 \text{ mod } 4 \\
     O/GL & \text{if } j \equiv 3 \text{ mod } 4
   \end{cases}
\end{equation*}
\end{center}
 
and
 
\begin{center}
$\mathcal{R}\Omega_{s}^{1} O/GL \cong GL/Sp$,
\end{center}
 
where $OGr$ is an "infinite orthogonal Grassmannian" and $HGr$ is an "infinite symplectic Grassmannian". As a consequence of all this, there is an isomorphism $[X, GL/Sp]_{\mathbb{A}^{1}} = GW_{1}^{3} (R)$. It is argued in \cite{AF2} that the morphisms of schemes $GL_{2n} \rightarrow A_{2n}$, $M \mapsto M^{t} \psi_{2n} M$ induce an isomorphism $GL/Sp \cong A$ of Nisnevich sheaves, where $A = colim_{n} A_{2n}$ (the transition maps are given by adding $\psi_{2}$). Altogether, we obtain an isomorphism $[X, A]_{\mathbb{A}^{1}} = GW_{1}^{3} (R)$ and $[X, A]_{\mathbb{A}^{1}}$ is precisely $A(R)/{\sim} = W'_E (R)$.\\
We describe an action of $\mathbb{G}_{m}$ on $GL/Sp$. For any ring $R$ and any unit $u \in R^{\times}$, we denote by $\gamma_{2n,u}$ the invertible $2n \times 2n$-matrix inductively defined by
 
\begin{center}
$\gamma_{2, u} =
\begin{pmatrix}
u & 0 \\
0 & 1
\end{pmatrix}
$
\end{center}
 
and $\gamma_{2n+2, u} = \gamma_{2n, u} \perp \gamma_{2, u}$. Conjugation by $\gamma_{2n, u}^{-1}$ induces an action of $\mathbb{G}_{m}$ on $GL_{2n}$ for all $n$. As $Sp_{2n}$ is preserved by this action, there is an induced action on $GL_{2n}/Sp_{2n}$. Since all the morphisms $GL_{2n}/Sp_{2n} \rightarrow GL_{2n+2}/Sp_{2n+2}$ are equivariant for this action, we obtain an action of $\mathbb{G}_{m}$ on $GL/Sp$.\\
Using the isomorphism $GL/Sp \xrightarrow{\sim} A$ described above, there is an induced action of $R^{\times} = \mathbb{G}_{m} (R)$ on $GW_{1}^{3} (R) \cong A(R)/{\sim} = W'_E (R)$ for any smooth $k$-algebra $R$ by taking $\mathbb{A}^{1}$-homotopy classes of morphisms. If $M \in GL_{2n} (R)$ represents a morphism $Spec (R) \rightarrow GL_{2n}$ and $u$ is a unit of $R$, note that the conjugation of $M$ by $\gamma_{2n, u}^{-1}$ is sent via the morphism $GL_{2n} \rightarrow GL_{2n}/Sp_{2n} \xrightarrow{\cong} A_{2n}$ to
 
\begin{center}
$\gamma_{2n, u}^{-1} M^{t} \gamma_{2n, u} \psi_{2n} \gamma_{2n, u} M \gamma_{2n, u}^{-1} = \gamma_{2n, u}^{-1} M^{t}(u \cdot \psi_{2n}) M \gamma_{2n, u}^{-1}$.
\end{center}
 
Note that the isometry induced by the matrix $\gamma_{2n, u}$ yields an equality
 
\begin{center}
$[R^{2n}, \psi_{2n}, \gamma_{2n, u}^{-1} M^{t}(u \cdot \psi_{2n}) M \gamma_{2n, u}^{-1}] = [R^{2n}, u \cdot \psi_{2n}, M^{t}(u \cdot \psi_{2n}) M]$
\end{center}
 
in $V (R)$. As a consequence, the action of $R^{\times}$ on $GW_{1}^{3} (R)$ can be described via the isomorphism $GW_{1}^{3} (R) \cong V (R)$ as follows: If $(P, g, f)$ is a triple as in the definition of the group $V (R)$ and $u \in R^{\times}$, then the action is given by
 
\begin{center}
$(u, (P, g, f)) \mapsto (P, u \cdot g, u \cdot f)$.
\end{center}
 
Following \cite{AF2}, we refer to this action as the conjugation action of $R^{\times}$ on $GW_{1}^{3} (R) \cong V (R)$. Recall the we have already defined an action of $R^{\times}$ on $V (R)$ for any commutative ring $R$ in Section \ref{3.2}. The conjugation action is thus a homotopy-theoretic interpretation of this action in case of a smooth algebra over a perfect field of characteristic $\neq 2$.\\
Now let us examine the product map
 
\begin{center}
$GW_{0}^{0} (R) \times GW_{1}^{3} (R) \rightarrow GW_{1}^{3} (R)$
\end{center}
 
for smooth $k$-algebras. As described above, there is a canonical isomorphism $GW_{0}^{0} = K_{0} O (R) = GW (R)$, where $GW (R)$ is the Grothendieck completion of the abelian monoid of non-degenerate symmetric bilinear forms over $R$. Furthermore, there is a canonical map

\begin{center}
$R^{\times} \rightarrow GW (R)$, $u \mapsto (R \times R \rightarrow R, (x,y) \mapsto uxy)$,
\end{center}

which induces an action of $R^{\times} = \mathbb{G}_{m} (R)$ on $GW_{1}^{3} (R)$ via the product map mentioned above. Again following \cite{AF2}, we refer to this action as the multiplicative action of $R^{\times} = \mathbb{G}_{m} (R)$ on $GW_{1}^{3} (R) \cong V (R)$. It follows from the proof of \cite[Proposition 3.5.1]{AF1} that the multiplicative action coincides with the conjugation action. Therefore we obtain another interpretation of the $R^{\times}$-action on $V (R)$ given in Section \ref{3.2} via the multiplicative structure of higher Grothendieck-Witt groups.

\section{Main results}\label{Main results}\label{4}

We finally give the definition of the generalized Vaserstein symbol in this section. As a first step, we recall the definition of the usual Vaserstein symbol introduced in \cite{SV} and reinterpret it by means of the isomorphism $W'_{E} (R) \cong V (R)$ discussed in the previous section. Then we define the generalized symbol and study its basic properties. In particular, we find criteria for the generalized Vaserstein symbol to be injective and surjective (onto the subgroup $\tilde{V} (R)$ of $V (R)$ corresponding to $W_{E} (R)$), which are the natural generalizations of the criteria found in \cite[Theorem 5.2]{SV}. These criteria will enable us to prove that the generalized Vaserstein symbol is a bijection e.g. for $2$-dimensional regular Noetherian rings and for $3$-dimensional regular affine algebras over algebraically closed fields such that $6 \in k^{\times}$.

\subsection{The Vaserstein symbol for unimodular rows}\label{The Vaserstein symbol for unimodular rows}\label{4.1}

For the rest of the section, let $R$ be a commutative ring. Let $\mathit{Um}_{3}(R)$ be its set of unimodular rows of length $3$, i.e. triples $a=(a_1,a_2,a_3)$ of elements in $R$ such that there are elements $b_1,b_2,b_3 \in R$ with $\sum_{i=1}^{3} a_i b_i = 1$. This data determines an exact sequence of the form

\begin{center}
$0 \rightarrow P(a) \rightarrow R^3 \xrightarrow{a} R \rightarrow 0$,
\end{center}

\noindent where $P(a) = \ker (a)$. The triple $b=(b_1,b_2,b_3)$ gives a section to the epimorphism $a:R^3 \rightarrow R$ and induces a retraction $r: R^3 \rightarrow P(a), e_{i} \mapsto e_{i} - a_{i} b$, where $e_{1} = (1,0,0)$, $e_{2} = (0,1,0)$ and $e_{3} = (0,0,1)$. One then obtains an isomorphism $i=r+a: R^3 \rightarrow P(a) \oplus R$, which induces an isomorphism $\det (R^{3}) \rightarrow \det (P(a) \oplus R)$. Finally, by composing with the canonical isomorphisms $\det (P(a) \oplus R) \cong \det (P(a))$ and $R \rightarrow \det (R^{3}), 1 \mapsto e_{1} \wedge e_{2} \wedge e_{3}$, one obtains an isomorphism $\theta: R \rightarrow \det (P(a))$.\\
The matrix

\begin{center}
$V (a,b) = \begin{pmatrix}
0 & - a_1 & - a_2 & - a_3 \\
a_1 & 0 & - b_3 & b_2 \\
a_2 & b_3 & 0 & - b_1 \\
a_3 & - b_2 & b_1 & 0
\end{pmatrix}$
\end{center}

\noindent has Pfaffian $1$ and its image in $W_E (R)$ does not depend on the choice of the section $b$. We therefore obtain a well-defined map $V: \mathit{Um}_{3}(R) \rightarrow W_E (R)$ called the Vaserstein symbol.\\
Now let us reinterpret the Vaserstein symbol map in light of the isomorphism $W'_E (R) \cong V(R)$ discussed in Section \ref{3.2}. The symbol $V(a)$ is sent to the element of $V (R)$ represented by the isometry class $[R^4, \psi_4, V (a,b)]$. If we denote by $\chi_a$ the alternating form $P (a) \times P (a) \rightarrow R, (p,q) \mapsto \theta^{-1} (p \wedge q)$, we obtain an alternating form on $R^4$ given by ${(i \oplus 1)}^{t} (\chi_a \perp \psi_2) {(i \oplus 1)}$. If we set
 
\begin{center}
$\sigma = \begin{pmatrix}
0 & 0 & 0 & - 1 \\
1 & 0 & 0 & 0 \\
0 & 1 & 0 & 0 \\
0 & 0 & 1 & 0
\end{pmatrix} \in E_4 (R)$,
\end{center}
 
then one can check that the form ${(i \oplus 1)}^{t} (\chi_a \perp \psi_2) {(i \oplus 1)}$ is given by the matrix $\sigma^{t} {V (a, b)}^{t} \sigma$. In particular, if we let $M: \mathit{Um}_{3}(R) \rightarrow \mathit{Um}_{3}(R)$ be the map which sends a unimodular row $a = (a_1 , a_2, a_3)$ to $M (a) = (-a_1, -a_2, -a_3)$, then the map $\nu \circ {V} \circ M$ is given by $a \mapsto [R^4, \psi_4, {(i \oplus 1)}^{t} (\chi_a \perp \psi_2) {(i \oplus 1)}]$. Since both $M$ and $\nu$ are bijections, these considerations lead to a generalization of the Vaserstein symbol.

\subsection{The generalized Vaserstein symbol}\label{The generalized Vaserstein symbol}

Now let $P_0$ be a projective $R$-module of rank $2$. We will use the notation of Section \ref{2.2}. For all $n \geq 3$, let $P_n = P_0 \oplus R e_3 \oplus ... \oplus R e_n$ be the direct sum of $P_0$ and free $R$-modules $R e_i$, $3 \leq i \leq n$, of rank $1$ with explicit generators $e_i$. We will sometimes omit these explicit generators in the notation. We denote by $\pi_{k, n}: P_n \rightarrow R$ the projections onto the free direct summands of rank $1$ with index $k = 3, ...,n$.\\
We assume that $P_0$ admits a trivialization $\theta_{0}: R \rightarrow \det(P_0)$ of its determinant. Then we denote by $\chi_0$ the non-degenerate alternating form on $P_0$ given by $P_0 \times P_0 \rightarrow R, (p,q) \mapsto \theta_{0}^{-1} (p \wedge q)$.\\
Now let $Um (P_0 \oplus R)$ be the set of epimorphisms $P_0 \oplus R \rightarrow R$. Any element $a$ of $Um (P_0 \oplus R)$ gives rise to an exact sequence of the form

\begin{center}
$0 \rightarrow P(a) \rightarrow P_0 \oplus R \xrightarrow{a} R \rightarrow 0$,
\end{center}

\noindent where $P(a) = \ker (a)$. Any section $s: R \rightarrow P_0 \oplus R$ of $a$ determines a canonical retraction $r: P_0 \oplus R \rightarrow P(a)$ given by $r(p)= p - s a(p)$ and an isomorphism $i: P_0 \oplus R \rightarrow P(a) \oplus R$ given by $i(p)  = a(p) + r(p)$.\\
The exact sequence above yields an isomorphism $\det(P_0) \cong \det (P(a))$ and therefore an isomorphism $\theta : R \rightarrow \det (P(a))$ obtained by composing with $\theta_0$. We denote by $\chi_a$ the non-degenerate alternating form on $P(a)$ given by $P(a) \times P(a) \rightarrow R, (p,q) \mapsto \theta^{-1} (p \wedge q)$.\\
We now want to define the generalized Vaserstein symbol
 
\begin{center}
$V_{\theta_{0}}: Um (P_0 \oplus R) \rightarrow V (R)$
\end{center}
 
associated to $P_{0}$ and the fixed trivialization $\theta_{0}$ of $\det (P_{0})$ by

\begin{center}
$V_{\theta_{0}} (a) = [P_0 \oplus R^2, \chi_0 \perp \psi_2, {(i \oplus 1)}^{t} (\chi_a \perp \psi_2) {(i \oplus 1)}]$.
\end{center}

If there is no ambiguity, we will usually suppress the fixed trivialization $\theta_{0}$ and denote $V_{\theta_{0}}$ simply by $V$ in order to simplify our notation. In order to prove that this generalized symbol is well-defined, one has to show that our definition is independent of a section of $a$:

\begin{Thm}\label{T4.1}
The generalized Vaserstein symbol is well-defined, i.e. the element $V (a)$ defined as above is independent of the choice of a section of $a$.
\end{Thm}

\begin{proof}
Let $a \in Um (P_0 \oplus R)$ with two sections $s, t: R \rightarrow P_0 \oplus R$. We denote by $i_s$ and $i_t$ the isomorphisms $P_0 \oplus R \cong P (a) \oplus R$ induced by the sections $s$ and $t$ respectively. Since the isomorphism $\det (P(a)) \cong \det (P_0)$ does not depend on the choice of a section (because the difference of two sections maps $R$ into $P(a)$), the form $\chi_{a}$ is independent of the choice of a section as well. Thus we have to show that the elements $V(a, s) = [P_0 \oplus R^2, (\chi_0 \perp \psi_2), {(i_{s} \oplus 1)}^{t} (\chi_{a} \perp \psi_2) {(i_{s} \oplus 1)}]$ and $V(a, t) = [P_0 \oplus R^2, (\chi_0 \perp \psi_2), {(i_{t} \oplus 1)}^{t} (\chi_{a} \perp \psi_2) {(i_{t} \oplus 1)}]$ are equal in $V (R)$.\\
We do this in the following three steps:\\
 
\begin{itemize}
\item We define a map $d : P_0 \oplus R \rightarrow R$. We get a corresponding automorphism $\varphi \in E (P_{0} \oplus R^2)$ defined by $\varphi = id_{P_0 \oplus R^2} - d e_4$;
\item We show that ${\varphi}^{t} {(i_{s} \oplus 1)}^{t} (\chi_{a} \perp \psi_2) {(i_{s} \oplus 1)} {\varphi} = {(i_{t} \oplus 1)}^{t} (\chi_{a} \perp \psi_2) {(i_{t} \oplus 1)}$;
\item Using Corollary \ref{C3.2}, we conclude that $V (a, s) = V (a, t)$.
\end{itemize}
 
Let us carry out the first step: We first define a map $d': P_0 \oplus R \rightarrow \det (P_0 \oplus R)$ by $p \mapsto s(1) \wedge t(1) \wedge p \in \det (P_0 \oplus R)$. Then $d: P_0 \oplus R \rightarrow R$ is the map obtained from $d'$ by composing with the isomorphisms $\det (P_0 \oplus R) \cong \det (P_0) \cong R$. Let $d_0$ and $d_{R}$ be its restrictions to $P_0$ and $R$ respectively. Furthermore, let $\varphi_0 = id_{P_0 \oplus R^2} - d_0 e_4$ and $\varphi_{R} = id_{P_0 \oplus R^2} - d_{R} e_4$ be the elementary automorphisms of $P_0 \oplus R^2$ defined by $- d_0$ and $- d_{R}$ respectively. Moreover, let $\varphi = id_{P_0 \oplus R^2} - d e_4$. Note that $\varphi = \varphi_0 \varphi_{R} = \varphi_{R} \varphi_0 \in E (P_{0} \oplus R^2)$.\\
Now let us conduct the second step. By Lemma \ref{L2.2}, we can check the desired equality locally. So let $\mathfrak{p}$ be a prime ideal of $R$ and $(e_{1}^{\mathfrak{p}}, e_{2}^{\mathfrak{p}})$ be a basis of the free $R_{\mathfrak{p}}$-module ${(P_0)}_{\mathfrak{p}}$ of rank $2$. We may further assume that ${({\theta}_{0}^{-1})}_{\mathfrak{p}} (e_{1}^{\mathfrak{p}} \wedge e_{2}^{\mathfrak{p}}) = 1$. With respect to the basis $(e_{1}^{\mathfrak{p}}, e_{2}^{\mathfrak{p}}, e_{3})$ of ${(P_0)}_{\mathfrak{p}} \oplus R_{\mathfrak{p}}$, the epimorphism $a_{\mathfrak{p}}$ can be represented by the unimodular row $(a_{1}^{\mathfrak{p}}, a_{2}^{\mathfrak{p}}, a_{3}^{\mathfrak{p}})$ and both sections $s_{\mathfrak{p}}$ and $t_{\mathfrak{p}}$ can be represented by the columns ${(s_{1}^{\mathfrak{p}}, s_{2}^{\mathfrak{p}}, s_{3}^{\mathfrak{p}})}^{t}$ and ${(t_{1}^{\mathfrak{p}}, t_{2}^{\mathfrak{p}}, t_{3}^{\mathfrak{p}})}^{t}$. Using the basis $(e_{1}^{\mathfrak{p}}, e_{2}^{\mathfrak{p}}, e_{3}, e_{4})$ of ${(P_0)}_{\mathfrak{p}} \oplus R_{\mathfrak{p}}^2$, we can check the desired equality locally: If we let $d_{1}^{\mathfrak{p}} = t_{3}^{\mathfrak{p}} s_{2}^{\mathfrak{p}} - t_{2}^{\mathfrak{p}} s_{3}^{\mathfrak{p}}$, $d_{2}^{\mathfrak{p}} = t_{1}^{\mathfrak{p}} s_{3}^{\mathfrak{p}} - t_{3}^{\mathfrak{p}} s_{1}^{\mathfrak{p}}$ and $d_{3}^{\mathfrak{p}} = t_{2}^{\mathfrak{p}} s_{1}^{\mathfrak{p}} - t_{1}^{\mathfrak{p}} s_{2}^{\mathfrak{p}}$ and

\begin{center}
$ M_{\mathfrak{p}} = \begin{pmatrix}
1 & 0 & 0 & 0 \\
0 & 1 & 0 & 0 \\
0 & 0 & 1 & 0 \\
- d_{1}^{\mathfrak{p}} & - d_{2}^{\mathfrak{p}} & -d_{3}^{\mathfrak{p}} & 1
\end{pmatrix}$
,
\end{center}

this amounts to verifying the equality

\begin{center}
$M_{\mathfrak{p}}^{t}
\begin{pmatrix}
0 & s_{3}^{\mathfrak{p}} & - s_{2}^{\mathfrak{p}} & a_{1}^{\mathfrak{p}} \\
- s_{3}^{\mathfrak{p}} & 0 & s_{1}^{\mathfrak{p}} & a_{2}^{\mathfrak{p}} \\
s_{2}^{\mathfrak{p}} & - s_{1}^{\mathfrak{p}} & 0 & a_{3}^{\mathfrak{p}} \\
- a_{1}^{\mathfrak{p}} & - a_{2}^{\mathfrak{p}} & - a_{3}^{\mathfrak{p}} & 0
\end{pmatrix}
M_{\mathfrak{p}}
$
=
$
\begin{pmatrix}
0 & t_{3}^{\mathfrak{p}} & - t_{2}^{\mathfrak{p}} & a_{1}^{\mathfrak{p}} \\
- t_{3}^{\mathfrak{p}} & 0 & t_{1}^{\mathfrak{p}} & a_{2}^{\mathfrak{p}} \\
t_{2}^{\mathfrak{p}} & - t_{1}^{\mathfrak{p}} & 0 & a_{3}^{\mathfrak{p}} \\
- a_{1}^{\mathfrak{p}} & - a_{2}^{\mathfrak{p}} & - a_{3}^{\mathfrak{p}} & 0
\end{pmatrix}
$
\end{center}

But this follows from the proof of \cite[Lemma 5.1]{SV}.\\
Finally, we conclude by Corollary \ref{C3.2}: Since $\varphi_{0}$ and $\varphi_{R}$ are elementary automorphisms of $P_{0} \oplus R^{2}$, the automorphism $\varphi = \varphi_{0} \varphi_{R}$ is an element of $E (P_{0} \oplus R^{2})$. By Corollary \ref{C3.2}, we deduce that
 
\begin{center}
$V (a,s) = [P_{0} \oplus R^{2}, \chi_{0} \perp \psi_{2}, {(i_{s} \oplus 1)}^{t} (\chi_{a} \perp \psi_{2}) {(i_{s} \oplus 1)}] \linebreak = [P_{0} \oplus R^{2}, \chi_{0} \perp \psi_{2}, {\varphi}^{t} {(i_{s} \oplus 1)}^{t} (\chi_{a} \perp \psi_{2}) {(i_{s} \oplus 1)} \varphi]$.
\end{center}
 
But by the second step, we also know that
 
\begin{center}
$[P_{0} \oplus R^{2}, \chi_{0} \perp \psi_{2}, {\varphi}^{t} {(i_{s} \oplus 1)}^{t} (\chi_{a} \perp \psi_{2}) {(i_{s} \oplus 1)} \varphi] \linebreak = [P_{0} \oplus R^{2}, \chi_{0} \perp \psi_{2}, {(i_{t} \oplus 1)}^{t} (\chi_{a} \perp \psi_{2}) {(i_{t} \oplus 1)}] = V (a,t)$.
\end{center}
 
This finishes the proof.
\end{proof}
 
We note that there is a homomorphism $\overline{\mathit{Pf}}: V (R) \rightarrow R^{\times}$ obtained as the composite $V(R) \xrightarrow{\cong} V_{\mathit{free}} (R) \xrightarrow{\xi} W'_E (R) \xrightarrow{\mathit{Pf}} R^{\times}$. We denote its kernel by $\tilde{V} (R)$. Of course, the isomorphism $V (R) \cong W'_E (R)$ induces an isomorphism $\tilde{V} (R) \cong W_E (R)$.\\
As stated in the previous section, the usual Vaserstein symbol of a unimodular row is an element of $W_{E} (R)$ and is invariant under elementary transformations. We will now prove that the analogous statements also hold for the generalized Vaserstein symbol:
 
\begin{Lem}\label{L4.2}
The generalized Vaserstein symbol $V: Um (P_0 \oplus R) \rightarrow V (R)$ maps $Um (P_0 \oplus R)$ into $\tilde{V} (R)$.
\end{Lem}
 
\begin{proof}
For this, we note that the Pfaffian of an element of $V (R)$ is completely determined by the Pfaffians of all its images under the maps $V (R) \rightarrow V (R_{\mathfrak{p}})$ induced by localization at any prime ideal $\mathfrak{p}$. But the localization ${(P_0)}_\mathfrak{p}$ at any prime $\mathfrak{p}$ is a free $R_\mathfrak{p}$-module of rank $2$; choosing a basis $({e}_{1}^{\mathfrak{p}}, {e}_{2}^{\mathfrak{p}})$ of ${(P_0)}_\mathfrak{p}$ such that ${({\theta}_{0}^{-1})}_{\mathfrak{p}} ({e}_{1}^{\mathfrak{p}} \wedge {e}_{2}^{\mathfrak{p}}) = 1$ as in the proof of Theorem \ref{T4.1}, we may calculate the Pfaffian of any Vaserstein symbol by the usual formula for the Pfaffian of an alternating $4\times4$-matrix. The lemma then follows immediately.
\end{proof}
 
\begin{Thm}\label{T4.3}
Let $\varphi$ be an elementary automorphism of $P_0 \oplus R$. Then we have $V (a) = V (a \varphi)$ for any $a \in Um (P_0 \oplus R)$. In particular, we obtain a well-defined map $V: Um (P_0 \oplus R)/E (P_0 \oplus R) \rightarrow \tilde{V} (R)$.
\end{Thm}
 
\begin{proof}
Let $\varphi$ be an elementary automorphism of $P_0 \oplus R$, $a \in Um (P_0 \oplus R)$ and $s: R \rightarrow P_0 \oplus R$ a section of $a$. Then ${\varphi}^{-1} s$ is a section of $a \varphi$. We let $i: P_0 \oplus R \rightarrow P(a) \oplus R$ and $j: P_0 \oplus R \rightarrow P(a \varphi) \oplus R$ be the isomorphisms induced by the sections $s$ and ${\varphi}^{-1} s$. We will show that

\begin{center}
${(\varphi \oplus 1)}^{t} {(i \oplus 1)}^{t} (\chi_a \perp \psi_2) {(i \oplus 1)} {(\varphi \oplus 1)} = {(j \oplus 1)}^{t} (\chi_{(a \varphi)} \perp \psi_2) {(j \oplus 1)}$.
\end{center}

The theorem then follows from Corollary \ref{C3.2}.\\
So let us show the equality above. Directly from the definitions, one checks that $(i \oplus 1) (\varphi \oplus 1) = ((\varphi \oplus 1) \oplus 1) (j \oplus 1)$, where by abuse of notation we understand $\varphi$ as the induced isomorphism $P (a \varphi) \rightarrow P(a)$. Altogether, it only remains to show that ${\varphi}^{t} \chi_a \varphi = \chi_{a \varphi}$.\\
For this, let $(p,q)$ a pair of elements in $P (a \varphi)$; by definition, $\chi_{a \varphi}$ sends these elements to the image of $p \wedge q$ under the isomorphism $\det (P (a \varphi)) \cong R$. This element can also be described as the image of $p \wedge q \wedge {\varphi}^{-1} s(1)$ under the isomorphism $\det (P_0 \oplus R) \cong R$.\\
Analogously, the alternating form ${\varphi}^{t} \chi_a \varphi$ sends $(p,q)$ to the image of the element ${\varphi} (p) \wedge {\varphi} (q) \wedge s(1)$ under the isomorphism $\det (P_0 \oplus R) \cong R$. Therefore Lemma \ref{L2.11} allows us to conclude as desired, which finishes the proof of the theorem.
\end{proof}
 
Note that if we equip the set $Um (P_0 \oplus R)$ with the projection $\pi_{R}: P_0 \oplus R \rightarrow R$ onto $R$ as a basepoint, then the generalized Vaserstein symbol is a map of pointed sets, because $V (\pi_{R}) = [P_0 \oplus R^2, \chi_0 \perp \psi_2, \chi_0 \perp \psi_2] = 0$.\\ \\
Let us briefly discuss how the generalized Vaserstein symbol depends on the choice of the trivialization $\theta_{0}$ of the determinant of $P_{0}$. For this, recall that we have defined an action of $R^{\times}$ on $V (R)$ in Section \ref{3.2}. In case of a smooth algebra over a perfect field of characteristic $\neq 2$, we saw in Section \ref{3.3} that this action can be identified with the multiplicative action induced by a product map in the theory of higher Grothendieck-Witt groups.\\
Now let $P_{0}$ be a projective $R$-module of rank $2$ which admits a trivialization $\theta_{0}$ of its determinant. Furthermore, let $a \in Um (P_{0} \oplus R)$ with section $s$ and let $i, \chi_{0}, \chi_{a}$ as in the definition of the generalized Vaserstein symbol. We consider another trivialization $\theta'_{0}$ of $\det (P_{0})$ and we let $\chi'_{0}$ and $\chi'_{a}$ be the corresponding alternating forms on $P_{0}$ and $P(a)$. Obviously, there is a unit $u \in R^{\times}$ such that $\theta_{0} = u \cdot \theta'_{0}$; in particular, we have $u \cdot \chi_{0} = \chi'_{0}$ and $u \cdot \chi_{a} = \chi'_{a}$. Thus, if we denote the Vaserstein symbol associated to $\theta'_{0}$ by $V_{\theta'_{0}}$, then
 
\begin{center}
$V_{\theta'_{0}} = [P_{0} \oplus R^{2}, (u \cdot \chi_{0}) \perp \psi_{2}, {(i \oplus 1)}^{t} ((u \cdot \chi_{a}) \perp \psi_{2}) (i \oplus 1)]$.
\end{center}
 
Finally, the isometry given by $P_{0} \oplus R^{2} \xrightarrow{id_{P_{0}} \oplus 1 \oplus u} P_{0} \oplus R^{2}$ yields an equality
 
\begin{center}
$[P_{0} \oplus R^{2}, (u \cdot \chi_{0}) \perp \psi_{2}, {(i \oplus 1)}^{t} ((u \cdot \chi_{a}) \perp \psi_{2}) (i \oplus 1)] \linebreak = [P_{0} \oplus R^{2}, u \cdot (\chi_{0} \perp \psi_{2}), u \cdot {(i \oplus 1)}^{t} (\chi_{a} \perp \psi_{2}) (i \oplus 1)]$.
\end{center}
 
Thus, if we denote the Vaserstein symbol associated to $\theta_{0}$ by $V_{\theta_{0}}$, then
 
\begin{center}
$V_{\theta'_{0}} = u \cdot V_{\theta_{0}}$.
\end{center}
 
In particular, the property of the generalized Vaserstein symbol to be injective, surjective or bijective onto $\tilde{V} (R)$ does not depend on the choice of $\theta_{0}$.\\
There is another immediate consequence of this: If we let $P_{0} = R^{2}$ be the free $R$-module of rank $2$ and let $e_{1} = (1,0), e_{2} = (0,1) \in R^{2}$ be the obvious elements, then there is a canonical isomorphism $\theta_{0}: R \xrightarrow{\cong} \det(R^{2})$ given by $1 \mapsto e_{1} \wedge e_{2}$. Then recall that the usual Vaserstein symbol can be described as $V_{\theta_{0}} \circ M$ (up to the identification $W_{E} (R) \cong \tilde{V} (R)$). Now let $a$ be a unimodular row of length $3$ over $R$ with section $b$ and let $V (a,b)$ the associated matrix mentioned in Section \ref{4.1}. By the formula above, it follows that $V_{-\theta_{0}} (a)$ is given by $[R^{4}, -\psi_{4},V(a,b)]$. But the matrix

\begin{center}
$\begin{pmatrix}
1 & 0 & 0 & 0 \\
0 & -1 & 0 & 0 \\
0 & 0 & 1 & 0 \\
0 & 0 & 0 & -1
\end{pmatrix}$
\end{center}

lies in $E_{4} (R)$ and gives an isometry between $\psi_{4}$ and $-\psi_{4}$. Hence the generalized Vaserstein symbol $V_{-\theta_{0}}$ associated to the trivialization $-\theta_{0}$ coincides with the usual Vaserstein symbol via the identification $\tilde{V} (R) \cong W_{E} (R)$ mentioned above.

\subsection{Criteria for the surjectivity and injectivity of the generalized Vaserstein symbol}\label{Criteria for the surjectivity and injectivity of the generalized Vaserstein symbol}

The main purpose of this section is to find some criteria for the generalized Vaserstein symbol to be surjective onto $\tilde{V} (R)$ or injective. We have already seen that these properties are independent of the choice of a trivialization of $\det (P_{0})$. So let us fix such a trivialization $\theta_{0}: R \xrightarrow{\cong} \det (P_{0})$.\\
Recall that a unimodular row of length $n$ is an $n$-tuple $a=(a_1,...,a_n)$ of elements in $R$ such that there are elements $b_1,...,b_n \in R$ with $\sum_{i=1}^{n} a_i b_i = 1$. We denote by $\mathit{Um}_{n}(R)$ the set of unimodular rows of length $n$. For any $n \geq 3$, there are obvious maps $U_{n}: \mathit{Um}_{n-2}(R) \rightarrow Um (P_n)$.\\
As a first step towards our criterion for the surjectivity of the generalized Vaserstein symbol (cf. Theorem \ref{T4.5} below), we prove the following statement:
 
\begin{Lem}\label{L4.4}
Any element of the form $[P_{4}, \chi_{0} \perp \psi_{2}, \chi] \in \tilde{V} (R)$ for a non-degenerate alternating form $\chi$ on $P_{4}$ is in the image of the generalized Vaserstein symbol.
\end{Lem}
 
\begin{proof}
First of all, we set $a = \chi (-, e_{4}): P_{0} \oplus R e_{3} \rightarrow R$. Since $\chi$ is non-degenerate, there is an element $p \in P_{4}$ such that $\chi (-,p): P_{4} \rightarrow R$ is just $-\pi_{4,4}$. In fact, since $\chi (p,p) = 0$, it immediately follows that $p \in P_{3}$. But then $a(p) = \chi (p, e_{4}) = - \chi (e_{4}, p) = 1$. Hence $p$ defines a section $s: R \rightarrow P_{3}$, $1 \mapsto p$, of $a: P_{0} \oplus Re_{3} \rightarrow R$.\\
The generalized Vaserstein symbol of $a$ may thus be computed by means of this section: As in the definition of the generalized Vaserstein symbol, we obtain an isomorphism $i: P_{0} \oplus R \rightarrow P (a) \oplus R$ and an alternating form $\chi_{a}$ on $P (a) = \ker (a)$ induced by $a$ and its section $s$. The generalized Vaserstein symbol of $a$ is then given by $[P_{0} \oplus R^2, \chi_{0} \perp \psi_{2}, {(i \oplus 1)}^{t} (\chi_{a} \perp \psi_{2}) {(i \oplus 1)}]$. But one can check easily that the form ${(i \oplus 1)}^{t} (\chi_{a} \perp \psi_{2}) {(i \oplus 1)}$ locally coincides with $\chi$ by construction. By Lemma \ref{L2.2}, it thus also coincides with $\chi$ globally. Therefore we obtain the desired equality $V (a) = [P_{0} \oplus R^2, \chi_{0} \perp \psi_{2}, \chi]$.
\end{proof}

Using Lemma \ref{L4.4} and the technical lemmas proven in previous sections, we may now prove the following criteria for the surjectivity of the generalized Vaserstein symbol:

\begin{Thm}\label{T4.5}
Let $N \in \mathbb{N}$. Assume that an element $\beta$ of $\tilde{V} (R)$ is of the form $[P_{2N+2}, \chi_0 \perp \psi_{2N}, \chi]$ for some non-degenerate alternating form on $P_{2N+2}$. Moreover, assume that $\pi_{2n+1, 2n+1} (E_{\infty} (P_{0}) \cap Aut (P_{2n+1})) = Um (P_{2n+1})$ for any $n \in \mathbb{N}$ with $1 < n \leq N$. Then $\beta$ lies in the image of the generalized Vaserstein symbol. Thus, the generalized Vaserstein symbol $V: Um (P_0 \oplus R) \rightarrow \tilde{V} (R)$ is surjective if $\pi_{2n+1, 2n+1} (E_{\infty} (P_{0}) \cap Aut (P_{2n+1})) = Um (P_{2n+1})$ for all $n \geq 2$.
\end{Thm}
 
\begin{proof}
By assumption, $\beta \in \tilde{V} (R)$ has the form $\beta = [P_{2N+2}, \chi_0 \perp \psi_{2N}, \chi]$ for some non-degenerate alternating form on $P_{2N+2}$. Furthermore, we may inductively apply Lemma \ref{L2.10} (because of the second assumption) in order to deduce that there is an elementary automorphism $\varphi$ on $P_{2N+2}$ such that ${\varphi}^{t} \chi \varphi = \psi \perp \psi_{2N-2}$ for some non-degenerate alternating form $\psi$ on $P_{4}$. In particular, $\beta = [P_{4}, \chi_{0} \perp \psi_{2}, \psi]$ by Corollary \ref{C3.2}. Finally, any element of this form is in the image of the generalized Vaserstein symbol by Lemma \ref{L4.4}. So $\beta$ is in the image of the generalized Vaserstein symbol.\\
For the last statement, note that any element of $\tilde{V} (R)$ is of the form $[R^{2n}, \psi_{2n}, \chi]$ for some non-degenerate alternating form on $R^{2n}$ (because of the isomorphism $\tilde{V} (R) \cong W_E (R)$). We may then artificially add a trivial summand $[P_0, \chi_{0}, \chi_{0}]$; hence any element of $\tilde{V} (R)$ is of the form $[P_{2n+2}, \chi_0 \perp \psi_{2n}, \chi_0 \perp \chi]$ for some non-degenerate alternating form on $R^{2n}$. We can then conclude by the previous paragraph.
\end{proof}
 
\begin{Thm}\label{T4.6}
Let $N \in \mathbb{N}$. Assume that the following conditions are satisfied:
 
\begin{itemize}
\item Every element of $\tilde{V} (R)$ is of the form $[R^{2N}, \psi_{2N}, \chi]$ for some non-degenerate alternating form on $R^{2N}$
\item One has $\pi_{2n+1, 2n+1} (E_{\infty} (P_{0}) \cap Aut (P_{2n+1})) = Um (P_{2n+1})$ for any $n \in \mathbb{N}$ with $1 < n < N$ and $U_{2N+1} (\mathit{Um}_{2N-1}(R)) \subset \pi_{2N+1, 2N+1} E (P_{2N+1})$
\end{itemize}
 
Then the generalized Vaserstein symbol $V: Um (P_0 \oplus R) \rightarrow \tilde{V} (R)$ is surjective.
\end{Thm}
 
\begin{proof}
We proceed as in the proof of Theorem \ref{T4.5}: By the first assumption, any element of $\tilde{V} (R)$ is of the form $[R^{2N}, \psi_{2N}, \chi]$ for some non-degenerate alternating form on $R^{2N}$. Again adding a trivial summand $[P_0, \chi_{0}, \chi_{0}]$, we see that any element of $\tilde{V} (R)$ is of the form $[P_{2N+2}, \chi_0 \perp \psi_{2N}, \chi_0 \perp \chi]$ for some non-degenerate alternating form on $R^{2N}$. As in the proof of Theorem \ref{T4.5}, it then follows inductively from Lemma \ref{L2.10} that any element of $\tilde{V} (R)$ is of the form $[P_0 \oplus R^2, \chi_{0} \perp \psi_{2}, \chi]$ for some non-degenerate alternating form $\chi$ on $P_0 \oplus R^2$. The generalized Vaserstein symbol is then surjective by Lemma \ref{L4.4}. Note that the condition $\pi_{2N+1, 2N+1} E (P_{2N+1}) = Um (P_{2N+1})$ can be replaced by the weaker condition $U_{2N+1} (\mathit{Um}_{2N-1}(R)) \subset \pi_{2N+1, 2N+1} E (P_{2N+1})$ in our situation.
\end{proof}
 
\begin{Kor}\label{C4.7}
Assume that the following conditions are satisfied:
 
\begin{itemize}
\item The usual Vaserstein symbol $V: \mathit{Um}_{3}(R) \rightarrow W_{E} (R)$ is surjective
\item $U_{5} (\mathit{Um}_{3}(R)) \subset \pi_{5, 5} (E_{\infty} (P_{0}) \cap Aut (P_{5}))$
\end{itemize}
 
Then the generalized Vaserstein symbol $V: Um (P_0 \oplus R) \rightarrow \tilde{V} (R)$ is surjective.
\end{Kor}
 
\begin{proof}
The surjectivity of the usual Vaserstein symbol means that any element of $\tilde{V} (R)$ is of the form $[R^4, \psi_{4}, \chi]$ for some non-degenerate alternating form on $R^4$. Now the corollary follows from Theorem \ref{T4.6}.
\end{proof}
 
In order to prove our criterion for the injectivity of the generalized Vaserstein symbol, we introduce the following condition: We say that $P_0$ satisfies condition $(\ast)$ if $[P_0 \oplus R^2, \chi_0 \perp \psi_2, \chi_1] = [P_0 \oplus R^2, \chi_0 \perp \psi_2, \chi_2] \in \tilde{V}(R)$ for alternating forms $\chi_1 , \chi_2$ on $P_0 \oplus R^2$ implies ${\alpha}^{t} (\chi_1 \perp \psi_{2n}) \alpha = \chi_2 \perp \psi_{2n}$ for some automorphism $\alpha \in E_{\infty} (P_0) \cap Aut (P_{2n+4})$.\\
If $P_0$ is a free $R$-module, condition $(\ast)$ is satisfied, which basically follows from the isomorphism $V (R) \cong W'_E (R)$. Furthermore, using the isomorphisms $V (R) \cong V_{\mathit{free}}(R) \cong W'_E (R)$, we will see that it is possible to prove that condition $(\ast)$ is always satisfied (cf. Lemma \ref{L4.9}).\\
As a first step towards Lemma \ref{L4.9}, we observe:

\begin{Lem}\label{L4.8}
Let $\chi$ be a non-degenerate alternating form on a finitely generated projective $R$-module $P$. Then there exists a finitely generated projective $R$-module $P'$ with a non-degenerate alternating form $\chi'$ on $P'$ and an isomorphism $\tau: R^{2n} \xrightarrow{\cong} P \oplus P'$ such that ${\tau}^{t}(\chi \perp  \chi') \tau = \psi_{2n}$.
\end{Lem}

\begin{proof}
Let $Q$ be a finitely generated projective $R$-module such that $P \oplus Q$ is free. Then, for $Q_{1} = P^{\vee} \oplus Q \oplus Q^{\vee}$, one has $P \oplus Q_{1} \cong R^{2m}$ for some $m \geq 0$. Moreover, for $\phi_{1} = can {\chi}^{-1} \perp H_{Q}$, the form $\chi \perp \phi_{1}$ is hence isometric to a form $\phi_{2}$ on $R^{2m}$. Now let $\phi_{3}$ be a form on $R^{2s}$ for some $s \geq 0$ which represents the inverse of $\phi_{2}$ in $W'_E (R)$. Then $\phi_{2} \perp \phi_{3} \perp \psi_{2t}$ is isometric to $\psi_{2m+2s+2t}$ for some $t \geq 0$. We set $P' = Q_{1} \oplus R^{2s+2t}$ and $\chi' = \phi_{1} \perp \phi_{3} \perp \psi_{2t}$. Then there is an isometry $\tau: R^{2m+2s+2n} \rightarrow P'$ between from $\psi_{2m+2s+2t}$ and $\chi \perp \chi'$, as desired. 
\end{proof}

Using Lemma \ref{L4.8}, we may prove:

\begin{Lem}\label{L4.9}
Any $P_{0}$ satisfies condition $(\ast)$.
\end{Lem}

\begin{proof}
We prove Lemma \ref{L4.10} below, which obviously implies Lemma \ref{L4.9} for $P = P_{0} \oplus R^{2}$ and $\chi = \chi_{0} \perp \psi_{2}$:
\end{proof}

\begin{Lem}\label{L4.10}
If $[P, \chi, \chi_{1}] = [P, \chi, \chi_{2}] \in V (R)$ for non-degenerate alternating forms $\chi$, $\chi_{1}$ and $\chi_{2}$ on a finitely generated projective $R$-module $P$, then we have an equality ${\alpha}^{t} (\chi_{1} \perp \psi_{2n}) \alpha = \chi_{2} \perp \psi_{2n}$ for some $n \in \mathbb{N}$ and some automorphism $\alpha \in E (P \oplus R^{2n})$.
\end{Lem}

\begin{proof}
The equality $[P, \chi, \chi_{1}] = [P, \chi, \chi_{2}]$ means that $[P, \chi_{1}, \chi_{2}] = 0$. By Lemma \ref{L4.8}, it follows that there is a finitely generated projective $R$-module $P_{1}$ with a non-degenerate alternating form $\chi'$ on $P_{1}$ and, moreover, with an isomorphism $\tau: R^{2m} \xrightarrow{\cong} P \oplus P_{1}$ such that ${\tau}^{t} (\chi_{1} \perp \chi') \tau = \psi_{2m}$. In particular, one has $0 = [P, \chi_{1}, \chi_{2}] = [R^{2m}, \psi_{2m}, {\tau}^{t}(\chi_{2} \perp \chi')\tau] \in V (R)$. Therefore the class of ${\tau}^{t}(\chi_{2} \perp \chi')\tau$ in $W'_E (R)$ is trivial and there exist $u \geq 1$ and $\zeta \in E (R^{2m+2u})$ such that ${\zeta}^{t} (({\tau}^{t} (\chi_{2} \perp \chi') \tau) \perp \psi_{2u}) \zeta = \psi_{2m+2u}$. Note that $\zeta$ lies in the commutator subgroup of $Aut (R^{2m+2u})$.\\
Again by Lemma \ref{L4.8}, there exists a finitely generated projective $R$-module $P_{2}$ with a non-degenerate alternating form $\chi''$ on $P_{2}$ and with an isomorphism $\beta: R^{2v} \xrightarrow{\cong} P_{1} \oplus R^{2u} \oplus P_{2}$ such that ${\beta}^{t} (\chi' \perp \psi_{2u} \perp \chi'') \beta = \psi_{2v}$.\\
But then the composite

\begin{center}
$\xi = (id_{P} \oplus {\beta}^{-1})(\tau \oplus id_{R^{2u}} \oplus id_{P_{2}})({\zeta}^{-1} \oplus id_{P_{2}})({\tau}^{-1} \oplus id_{R^{2u}} \oplus id_{P_{2}})(id_{P} \oplus \beta)$
\end{center}

is an isometry from $\chi_{1} \perp \psi_{2v}$ to $\chi_{2} \perp \psi_{2v}$ and lies in the commutator subgroup of $Aut (P \oplus R^{2v})$ because it is a conjugate of ${\zeta}^{-1} \perp id_{P_{2}}$. In particular, it follows that $\xi \perp id_{R^{2w}} \in E (P \oplus R^{2v+2w})$ for some $w \geq 0$. Finally, if we then set $\alpha = \xi \perp id_{R^{2w}}$ and $n = v+w$, the lemma is proven.
\end{proof}

Now that we have proven that condition $(\ast)$ is always satisfied, we can find conditions which imply that two elements $a,b \in Um (P_{0} \oplus R)$ with the same Vaserstein symbol are equal up to a stably elementary automorphism of $P_{0} \oplus R$. More precisely:

\begin{Thm}\label{T4.11}
Assume that $E (P_{2n}) e_{2n} = (E_{\infty} (P_0) \cap Aut (P_{2n})) e_{2n}$ for $n \geq 2$. Then the equality $V (a) = V (b)$ for $a,b \in Um (P_0 \oplus R)$ implies that $b = a \varphi$ for some $\varphi \in E_{\infty} (P_0) \cap Aut (P_{3})$.
\end{Thm}

\begin{proof}
Let $a$ and $b$ elements of $Um (P_{0} \oplus R)$ with sections $s$ and $t$ respectively and let $i: P_{0} \oplus R \rightarrow P (a) \oplus R$ and $j: P_{0} \oplus R \rightarrow P (a) \oplus R$ be the isomorphisms induced by these sections. Furthermore, we let $V (a, s) = {(i \oplus 1)}^{t} (\chi_{a} \perp \psi_{2}) {(i \oplus 1)}$ and $V (b, t) = {(j \oplus 1)}^{t} (\chi_{b} \perp \psi_{2}) {(j \oplus 1)}$ be the non-degenerate alternating forms on $P_{0} \oplus R^2$ appearing in the definition of the generalized Vaserstein symbols of $a$ and $b$ respectively. Now assume that $V (a) = V (b)$. Since $P_{0}$ satisfies condition $(\ast)$, there exist $n \in \mathbb{N}$ and an automorphism $\alpha \in E_{\infty} (P_{0}) \cap Aut (P_{2n+4})$ such that ${\alpha}^{t} (V (a, s) \perp \psi_{2n}) \alpha = V (b, t) \perp \psi_{2n}$. Using Lemma \ref{L2.9}, we inductively deduce that ${\beta}^{t} V (a, s) \beta = V (b, t)$ for some $\beta \in E_{\infty} (P_{0}) \cap Aut (P_{0} \oplus R^2)$. Now by Lemma \ref{L2.8} and the second assumption in the theorem, there exists an automorphism $\gamma \in E (P_{0} \oplus R^2) \cap Sp (V (a, s))$ such that $\beta e_{4} = \gamma e_{4}$.\\
We now define $\delta: P_{0} \oplus R \rightarrow P_{0} \oplus R$ as the composite
 
\begin{center}
$P_{0} \oplus R e_{3} \rightarrow P_{0} \oplus R e_{3} \oplus R e_{4} \xrightarrow{{\gamma}^{-1} \beta} P_{0} \oplus R e_{3} \oplus R e_{4} \rightarrow P_{0} \oplus R e_{3}$.
\end{center}
 
One can then check that $\delta$ is an element of $E_{\infty} (P_{0}) \cap Aut (P_{0} \oplus R)$. Moreover, we have
 
\begin{center}
${\beta}^{t} {({\gamma}^{-1})}^{t} V (a, s) {\gamma}^{-1} \beta = V (b, t)$
\end{center}
 
and in particular $a \delta = b$, as desired.
\end{proof}

\begin{Kor}\label{C4.12}
Under the hypotheses of Theorem \ref{T4.11}, furthermore assume that $a (E_{\infty} (P_0) \cap Aut (P_0 \oplus R)) = a E (P_0 \oplus R)$ for all $a \in Um (P_0 \oplus R)$. Then the generalized Vaserstein symbol $V: Um (P_0 \oplus R)/E (P_0 \oplus R) \rightarrow \tilde{V} (R)$ is injective.
\end{Kor}
 
\begin{proof}
By Theorem \ref{T4.11}, we have that $V (a) = V (b)$ implies $b = a {\varphi'}$ for some ${\varphi'} \in E_{\infty} (P_0) \cap Aut (P_0 \oplus R)$. Now by the additional assumption, there also exists an elementary automorphism $\varphi$ of $P_0 \oplus R$ such that $b = a \varphi$. So the generalized Vaserstein symbol is injective.
\end{proof}

Regarding the additional assumption in Corollary \ref{C4.12}, it is possible to adapt the arguments in the proof of \cite[Corollary 7.4]{SV} to show that the desired equality $a (E_{\infty} (P_0) \cap Aut (P_0 \oplus R)) = a E (P_0 \oplus R)$ holds for all $a \in Um (P_0 \oplus R)$ if $E_{\infty} (P_{0}) \cap Aut (P_{0} \oplus R^{2}) = E (P_{4})$:

\begin{Lem}\label{L4.13}
If $E_{\infty} (P_{0}) \cap Aut (P_{0} \oplus R^{2}) = E (P_{4})$, then we have an equality $a (E_{\infty} (P_0) \cap Aut (P_0 \oplus R)) = a E (P_0 \oplus R)$ for all $a \in Um (P_0 \oplus R)$.
\end{Lem}

\begin{proof}
Let $a \in Um (P_0 \oplus R)$ with section $s$ and let $\varphi \in E_{\infty} (P_0) \cap Aut (P_0 \oplus R)$. If we let $V (a,s)$ be the alternating form from the definition of the generalized Vaserstein symbol, then it follows from the proof of Lemma \ref{L4.4} that

\begin{center}
${(\varphi \oplus 1)}^{t} V (a,s) {(\varphi \oplus 1)} = V (a', s')$
\end{center}

for some $a' \in Um (P_{0} \oplus R)$ with section $s'$. By assumption, the automorphism $\varphi \oplus 1$ of $P_{4}$ is an elementary automorphism. Moreover, by Corollary \ref{C2.4}, the group $E (P_{4})$ is generated by elementary automorphisms $\varphi_{g} = id_{P_{4}} + g$, where $g$ is a homomorphism

\begin{itemize}
\item[1)] $g: R e_{3} \rightarrow P_{0}$,
\item[2)] $g: P_{0} \rightarrow R e_{3}$,
\item[3)] $g: R e_{3} \rightarrow R e_{4}$ or
\item[4)] $g: R e_{4} \rightarrow R e_{3}$.
\end{itemize}

It therefore suffices to show the following: If $\varphi_{g}^{t} V (a,s) \varphi_{g} = V (a',s')$ for some $g$ as above, then $a' = a \psi$ for some $\psi \in E (P_{0} \oplus R)$. The only non-trivial case is the last one, i.e. if $g$ is a homomorphism $R e_{4} \rightarrow R e_{3}$.\\
So let $g: R e_{4} \rightarrow R e_{3}$ and let $\varphi_{g}$ be the induced elementary automorphism of $P_{4}$. As explained above, we assume that

\begin{center}
$\varphi_{g}^{t} V (a,s) \varphi_{g} = V (a', s')$
\end{center}

for some epimorphism $a': P_{0} \oplus R e_{3} \rightarrow R$ with section $s'$. Write $a = (a_{0}, a_{R})$, where $a_{0}$ and $a_{R}$ are the restrictions of $a$ to $P_{0}$ and $R e_{3}$ respectively. Furthermore, let $p = \pi_{P_{0}} (s (1))$. From now on, we interpret the alternating form $\chi_{0}$ in the definition of the generalized Vaserstein symbol as an alternating isomorphism $\chi_{0}: P \rightarrow P^{\vee}$. Then one can check locally that

\begin{center}
$a' = (a_{0} - g(1) \cdot \chi_{0} (p), a_{R})$.
\end{center}

Let us define an elementary automorphism $\psi$ as follows: We first define an endomorphism of $P_{0}$ by 

\begin{center}
$\psi_{0} = id_{P_{0}} - g(1) \cdot \pi_{P_{0}} \circ s \circ \chi_{0} (p) : P_{0} \rightarrow P_{0}$
\end{center}
 
and we also define a morphism $P_{0} \rightarrow R e_{3}$ by
 
\begin{center}
$\psi_{R} = - g(1) \cdot \pi_{R} \circ s \circ \chi_{0} (p): P_{0} \rightarrow R$.
\end{center}

Then we consider the endomorphism of $P_{0} \oplus R$ given by
 
\begin{center}
$\psi =
\begin{pmatrix}
\psi_{0} & 0 \\
\psi_{R} & id_R
\end{pmatrix}
$.
\end{center}

First of all, this endomorphism coincides up to an elementary automorphism with

\begin{center}
$
\begin{pmatrix}
\psi_{0} & 0 \\
0 & id_R
\end{pmatrix}
$.
\end{center}

Since $\chi_{0} (p) \circ \pi_{P_{0}} \circ s = 0$, this endomorphism is an element of $E (P_{0} \oplus R)$ by Lemma \ref{L2.6}. Hence the same holds for $\psi$. Finally, one can check easily that $a \psi = a'$ by construction.
\end{proof}

As an immediate consequence, we can finally deduce our criterion for the injectivity of the generalized Vaserstein symbol:

\begin{Thm}\label{T4.14}
Assume that $E (P_{2n}) e_{2n} = (E_{\infty} (P_{0}) \cap Aut (P_{2n})) e_{2n}$ for all $n \geq 3$ and $E_{\infty} (P_{0}) \cap Aut (P_{4}) = E (P_{4})$. Then the generalized Vaserstein symbol $V: Um (P_{0} \oplus R)/E (P_{0} \oplus R) \rightarrow \tilde{V} (R)$ is injective.
\end{Thm}

\begin{proof}
Combine Corollary \ref{C4.12} and Lemma \ref{L4.13}.
\end{proof}

\subsection{The bijectivity of the generalized Vaserstein symbol in dimension 2 and 3}

Let us now study the criteria for the surjectivity and injectivity of the generalized Vaserstein symbol found in this section. In \cite{HB} the conditions of Theorem \ref{T4.5} and Theorem \ref{T4.14} are studied in a very general framework. If $R$ is a Noetherian ring of Krull dimension $d$, it follows from \cite[Chapter IV, Theorem 3.4]{HB} that actually $Unim.El. (P_{n}) = E (P_{n}) e_{n}$ for all $n \geq d+2$ (or $Um (P_{n}) = \pi_{n,n} E (P_{n})$ for all $n \geq d+2$). In particular, if $\dim (R) \leq 4$, then the generalized Vaserstein symbol is injective as soon as $E_{\infty} (P_{0}) \cap Aut (P_{4}) = E (P_{4})$; if $\dim (R) \leq 3$, it is surjective. Hence the following results are immediate consequences of our stability results in Section \ref{2.3}:

\begin{Thm}\label{T4.15}
Assume that $R$ is either a regular Noetherian ring of dimension $2$ or a regular affine algebra of dimension $3$ over a perfect field $k$ such that $c.d.(k) \leq 1$ and moreover $6 \in k^{\times}$. Then the generalized Vaserstein symbol $V: Um (P_{0} \oplus R)/E (P_{0} \oplus R) \rightarrow \tilde{V} (R)$ is a bijection.
\end{Thm}

\begin{Thm}\label{T4.16}
Let $R$ be a $4$-dimensional regular affine algebra over a perfect field $k$ satisfying the property $\mathcal{P} (5,3)$ (cf. Section \ref{2.3}). Then the generalized Vaserstein symbol $V: Um (P_{0} \oplus R)/E (P_{0} \oplus R) \rightarrow \tilde{V} (R)$ is injective.
\end{Thm}

Because of the pointed surjection $Um (P_{0} \oplus R)/ E (P_{0} \oplus R) \rightarrow \phi_{2}^{-1} ([P_{0} \oplus R])$, the bijectivity of the generalized Vaserstein symbol always gives rise to a surjection $W_E (R) \rightarrow \phi_{2}^{-1} ([P_{0} \oplus R])$; in this case, it seems that the group structure of $W_E (R) \cong Um (P_{0} \oplus R)/ E (P_{0} \oplus R)$ essentially governs the structure of the fiber $\phi_{2}^{-1} ([P_{0} \oplus R])$.\\
The following application follows - to some degree - the pattern of the proof of \cite[Theorem 7.5]{FRS} and illustrates the previous paragraph:
 
\begin{Thm}\label{T4.17}
Let $R$ be a ring and $P_0$ be a projective $R$-module of rank $2$ which admits a trivialization $\theta_{0}$ of its determinant. Assume the following conditions are satisfied:
\begin{itemize}
\item[a)] The generalized Vaserstein symbol $V: Um (P_{0} \oplus R)/ E (P_{0} \oplus R) \rightarrow \tilde{V} (R)$ induced by $\theta_{0}$ is a bijection;
\item[b)] $2 V (a_{0}, a_{R}) = V (a_{0}, a_{R}^{2})$ for $(a_{0}, a_{R}) \in Um (P_{0} \oplus R)$;
\item[c)] The group $W_E (R)$ is $2$-divisible.
\end{itemize}
 
Then $\phi_{2}^{-1} ([P_{0} \oplus R])$ is trivial.
\end{Thm}
 
\begin{proof}
Assume $P' \oplus R \cong P_{0} \oplus R$. As we have seen in Section \ref{2.4}, $P'$ has an associated element of $Um (P_{0} \oplus R)/ Aut (P_{0} \oplus R)$. We lift this element to an element $[b]$ of $Um (P_{0} \oplus R)/ E (P_{0} \oplus R)$ ($[b]$ denotes the class of $b \in Um (P_{0} \oplus R)$). Since the generalized Vaserstein symbol is a bijection and $W_E (R)$ is a $2$-divisible group by assumption, we get that $[b] = 2 [a]$, where $[a]$ denotes the class of an element $a = (a_{0}, a_{R})$ of $Um (P_{0} \oplus R)$ in the orbit space $Um (P_{0} \oplus R)/ E (P_{0} \oplus R)$. But then the second assumption shows that $2 [a] = [(a_{0}, a_{R}^{2})]$. It follows from \cite[Proposition 2.7]{B} or \cite[Lemma 2]{S1} that any element of $Um (P_{0} \oplus R)$ of the form $(a_{0}, a_{R}^{2})$ is completable to an automorphism of $P_{0} \oplus R$, i.e. $\pi_{R} = a \varphi$ for some automorphism $\varphi$ of $P_{0} \oplus R$. Altogether, $\pi_{R}$ and $b$ therefore lie in the same orbit under the action of $Aut (P_{0} \oplus R)$ and hence $P' \cong P$. Thus, $\phi_{2}^{-1} ([P_{0} \oplus R])$ is trivial.
\end{proof}
 
As mentioned in the proof of Theorem \ref{T4.16}, any element $a \in Um (P_{0} \oplus R)$ of the form $a = (a_{0},a_{R}^{2})$ is completable to an automorphism of $P_{0} \oplus R$. This follows directly from \cite[Proposition 2.7]{B} or \cite[Lemma 2]{S1}, because $P_{0}$ has a trivial determinant. We now construct a more concrete completion of $a = (a_{0},a_{R}^{2})$. For this, let us first look at the case $P_{0} \cong R^{2}$: If $(b, c, a^2)$ is a unimodular row and $qb + rc + ap = 1$, then it follows from \cite{Kr} that the matrix
 
\begin{center}
$ 
\begin{pmatrix}
-p-qr & q^{2} & -c+2aq \\
-r^{2} & -p+qr & b+2ar \\
b & c & a^{2}
\end{pmatrix}
$
\end{center}
 
is a completion of $(b, c, a^2)$ with determinant $1$. We observe that
 
\begin{center}
$ 
\begin{pmatrix}
-qr & q^{2} \\
-r^{2} & qr
\end{pmatrix}
$
=
$ 
\begin{pmatrix}
q \\
r
\end{pmatrix}
$
$ 
\begin{pmatrix}
-r & q
\end{pmatrix}
$
\end{center}
 
and also
 
\begin{center}
$ 
\begin{pmatrix}
-c \\
b
\end{pmatrix}
$
=
$ 
\begin{pmatrix}
0 & -1\\
1 & 0
\end{pmatrix}
$
$ 
\begin{pmatrix}
b \\
c
\end{pmatrix}
$.
\end{center}

This shows how to generalize the construction of this explicit completion. We denote by $\chi_{0}: P_{0} \rightarrow P_{0}^{\vee}$ the alternating isomorphism from the definition of the generalized Vaserstein symbol (we now interpret it as an alternating isomorphism and not as a non-degenerate alternating form). If $a = (a_{0}, a_{R})$ is an element of $Um (P_{0} \oplus R)$ with a section $s$ uniquely given by the element $s (1) = (q, p) \in P_{0} \oplus R$, we consider the following morphisms: We define an endomorphism of $P_0$ by
 
\begin{center}
$\varphi_{0} = -(\pi_{P_{0}} s) \circ \chi_{0} (q) - p \cdot id_{P_{0}} : P_{0} \rightarrow P_{0}$
\end{center}
 
and we also define a morphism $R \rightarrow P_{0}$ by
 
\begin{center}
$\varphi_{R}: R \rightarrow P_{0}$, $1 \mapsto 2 a_{R}(1) \cdot q + \chi_{0}^{-1} (a_{0})$.
\end{center}
 
Then we consider the endomorphism of $\varphi: P_{0} \oplus R$ given by
 
\begin{center}
$ 
\begin{pmatrix}
\varphi_{0} & \varphi_{R} \\
a_{0} & a_{R}^{2}
\end{pmatrix}
$.
\end{center}

Essentially by construction, $\varphi$ is a completion of $(a_{0},a_{R}^{2})$:
 
\begin{Prop}\label{P4.18}
The endomorphism $\varphi$ of $P_{0} \oplus R$ defined above is an automorphism of $P_{0} \oplus R$ of determinant $1$ such that $\pi_{R} \varphi = (a_{0}, a_{R}^{2})$.
\end{Prop}

\begin{proof}
Choosing locally a free basis $(e_{1}^{\mathfrak{p}}, e_{2}^{\mathfrak{p}})$ of ${(P_{0})}_{\mathfrak{p}}$ at any prime $\mathfrak{p}$ such that ${(\theta_{0}^{-1})}_{\mathfrak{p}} (e_{1}^{\mathfrak{p}} \wedge e_{2}^{\mathfrak{p}}) = 1$, we can check locally that this endomorphism is an automorphism of determinant $1$ (because locally it coincides with the completion given in \cite{Kr}); by definition, we also have $\pi_{R} \varphi = (a_{0}, a_{R}^{2})$. Thus, $\varphi$ has the desired properties and generalizes the explicit completion given in \cite{Kr}.
\end{proof}

\end{document}